\documentclass[11pt,a4paper]{article}
\usepackage{amsmath}
\usepackage{mathrsfs}
\usepackage{graphicx,color}
\usepackage{bm}
\usepackage{amssymb}
\usepackage{theorem}
\usepackage{amsfonts}
\usepackage{here} 
\usepackage{enumerate}
\usepackage[font=footnotesize]{caption} 
\numberwithin{equation}{section}

\newcommand{\tendsto}{\to}
\newcommand{\indic}{\mathbf{1}}
\newcommand{\qed}{\hfill$\Box$}
\renewcommand{\Pr}{\operatorname{P}}

\newcommand{\E}{\operatorname{E}}
\newcommand{\Var}{\operatorname{Var}}
\newcommand{\Cov}{\operatorname{Cov}}
\renewcommand{\d}{\mathrm{d}}

\newcommand{\oP}{o_{\scriptscriptstyle P}}
\newcommand{\OP}{O_{\scriptscriptstyle P}}

\newcommand{\vecR}{\boldsymbol{R}}

\newcommand{\vecU}{\boldsymbol{U}}
\newcommand{\vecV}{\boldsymbol{V}}
\newcommand{\vecX}{\boldsymbol{X}}

\newcommand{\vecm}{\boldsymbol{m}}

\newcommand{\vecu}{\boldsymbol{u}}
\newcommand{\vecv}{\boldsymbol{v}}
\newcommand{\vecw}{\boldsymbol{w}}
\newcommand{\vecx}{\boldsymbol{x}}
\newcommand{\vecy}{\boldsymbol{y}}

\renewcommand{\tilde}{\widetilde}
\renewcommand{\hat}{\widehat}

\renewcommand{\epsilon}{\varepsilon}
\newcommand{\eps}{\varepsilon}

\newcommand{\bbC}{\mathbb{C}}

\newcommand{\bbG}{\mathbb{G}}

\newcommand{\bbN}{\mathbb{N}}

\newcommand{\bbR}{\mathbb{R}}

\newcommand{\bbU}{\mathbb{U}}

\newcommand{\tilvecV}{\tilde{\vecV}}

\title{The Empirical Beta Copula}
\author{Johan Segers\footnote{Institut de Statistique, Biostatistique et Sciences Actuarielles, 
Universit\'{e} catholique de Louvain, Voie du Roman Pays 20, B-1348 Louvain-la-Neuve, Belgium,
E-mail: \texttt{johan.segers@uclouvain.be}}, 
\ Masaaki Sibuya\footnote{Professor Emeritus, Keio University, E-mail: 
\texttt{sibuyam@1986.jukuin.keio.ac.jp}}, 
\ and Hideatsu Tsukahara\footnote{Corresponding author. Faculty of Economics, Seijo University, 6--1--20 Seijo, Setagaya-ku, 
Tokyo, 157-8511, Japan, E-mail: \texttt{tsukahar@seijo.ac.jp}, Tel: +81-3-3482-9264}}
\date{\today}
    {\begin{itemize}%
     }%
    {\end{itemize}}

\newenvironment{proof}[1]{\par\medskip\noindent\textit{#1.}}{\qed\par\medskip}

\theoremstyle{plain}
\theorembodyfont{\upshape}      
\newtheorem{Def}{Definition}[section]
\newtheorem{Rem}[Def]{Remark}

\newtheorem{Cond}[Def]{Condition}

\theoremstyle{break}

\theoremstyle{plain}
\theorembodyfont{\slshape}
\newtheorem{Lem}[Def]{Lemma}
\newtheorem{Prop}[Def]{Proposition}
\newtheorem{Cor}[Def]{Corollary}
\newtheorem{Thm}[Def]{Theorem}

\theoremstyle{break}

\definecolor{brown}{rgb}{0.59, 0.29, 0.0}

\begin{document}
\maketitle
\begin{abstract}
Given a sample from a continuous multivariate distribution $F$, the uniform random variates generated independently and rearranged in the order specified by the componentwise ranks of the original sample look like a sample from the copula of $F$. This idea can be regarded as a variant on Baker's [J. Multivariate Anal.\ 99 (2008) 2312--2327] copula construction and leads to the definition of the empirical beta copula.  The latter turns out to be a particular case of the empirical Bernstein copula, the degrees of all Bernstein polynomials being equal to the sample size.  

Necessary and sufficient conditions are given for a Bernstein polynomial to be a copula. These imply that the empirical beta copula is a genuine copula. 
Furthermore, the empirical process based on the empirical Bernstein copula is shown to be asymptotically the same as the ordinary empirical copula process under assumptions which are significantly weaker than those given in Janssen, Swanepoel and Veraverbeke [J. Stat.\ Plan.\ Infer.\ 142 (2012) 1189--1197]. 

A Monte Carlo simulation study shows that the empirical beta copula outperforms the empirical copula and the empirical checkerboard copula in terms of both bias and variance.  Compared with the empirical Bernstein copula with the smoothing rate suggested by Janssen et al., its finite-sample performance is still significantly better in several cases, especially in terms of bias.

\bigskip\noindent
\textsl{AMS 2010 subject classifications}: 62G20, 62G30, 62H12\\
\textsl{Keywords}: Copula, Empirical copula, Bernstein polynomial, Empirical Bernstein copula, Checkerboard copula

\end{abstract}

\section{Introduction}

Let $\vecX_i=(X_{i,1},\ldots,X_{i,d})$, $i\in\{1,\ldots,n\}$, be independent and identically 
distributed random vectors, and assume that the cumulative distribution function, $F$, of $\vecX_i$
is continuous.  By Sklar's theorem (Sklar~\cite{Sklar59}), there exists a unique copula, $C$, such that
\[
  F(x_1,\ldots,x_d)=C\{F_1(x_1),\ldots,F_d(x_d)\},
\]
where $F_j$ is the $j$th marginal distribution function of $F$.
For $i\in\{1,\ldots,n\}$ and $j\in\{1,\ldots,d\}$, let $R_{i,j}^{(n)}$ be the rank of $X_{i,j}$ 
among $X_{1,j},\ldots,X_{n,j}$; namely, 
\begin{equation}
\label{eq:ranks}
  R_{i,j}^{(n)}
  = \sum_{k=1}^n \indic\{ X_{k,j} \le X_{i,j}\}.
\end{equation}
The vector of ranks is denoted by $\vecR_i^{(n)}:=(R_{i,1}^{(n)},\ldots,R_{i,d}^{(n)})$.
The  basic nonparametric estimator for the copula $C$ is 
the empirical copula (Deheuvels~\cite{Deheu79}), and we use the version given by
\begin{equation}
  \label{eq:empirical-copula}
  \bbC_n(\vecu):=\frac{1}{n}\sum_{i=1}^n\prod_{j=1}^d\indic\biggl\{
  \frac{R_{i,j}^{(n)}}{n}\leq u_j\biggr \}, \quad \vecu=(u_1,\ldots,u_d)\in [0,1]^d.
\end{equation}
Slightly different definitions are employed for instance in Fermanian et al.~\cite{Fer-Rad-Weg04} and Tsukahara~\cite{Tsuka05}, but the supremum distance between these empirical copula variants is at most $d/n$.  In Section~\ref{sec:simul}, we will also use the so-called empirical checkerboard copula; see \eqref{eq:emp-checker-copula}.

Let $W_{1,1},\ldots,W_{n,1};\ldots;W_{1,d},\ldots,W_{n,d}$ be independent random variables 
which are uniformly distributed on the unit 
interval and are independent of the sample $\vecX_1, \ldots, \vecX_n$. 
For each $j\in\{1,\ldots,d\}$, let $V_{1,j}^{(n)}<\cdots <V_{n,j}^{(n)}$ be 
the order statistics based on $W_{1,j},\ldots,W_{n,j}$.  Define, for $i\in\{1,\ldots,n\}$, 
\begin{equation}
  \label{eq:V-i}
  \tilvecV_i^{(n)}:=\Bigl(V_{R_{i,1}^{(n)},1}^{(n)},\ldots,V_{R_{i,d}^{(n)},d}^{(n)}\Bigr), 
\end{equation}
One of the authors conceived that $\tilvecV_1^{(n)},\ldots,\tilvecV_n^{(n)}$ could be interpreted
as a sample from some version of the empirical copula.  Although this was not quite correct, 
we found that picking one vector randomly from these $n$ vectors is equivalent to 
sampling from a smoothed version of the empirical copula, which we call 
the \emph{empirical beta copula} in view of the marginal distribution of uniform order statistics.  This idea may be regarded as 
Baker~\cite{Baker2008}'s copula construction based on uniform order statistics,
where $d$-tuples of order statistics are determined by the rank vectors $\vecR_1^{(n)},\ldots,
\vecR_n^{(n)}$, and the probability of choosing among them is equal to $1/n$. 

The empirical beta copula arises as a particular case of the empirical Bernstein 
copula when the degrees of all Bernstein polynomials are set equal to the sample size. 
The Bernstein copula and the empirical Bernstein copula are introduced in Sancetta and 
Satchell~\cite{Sance-Satch2004}, and the asymptotic behavior of the latter is studied 
in Janssen et al.~\cite{Jan-Swan-Ver2012}.
We give necessary and sufficient conditions for the empirical Bernstein copula 
to be a genuine copula, conditions which hold for the empirical beta copula.
We show asymptotic results for the empirical Bernstein copula process under weaker assumptions
than those in Janssen et al.~\cite{Jan-Swan-Ver2012}.

The advantages of the empirical beta copula are that it is a genuine copula, that it does not require the choice of a smoothing parameter, and that simulating random samples from it is straightforward.  Furthermore, the corresponding empirical process converges weakly to a Gaussian process under standard smoothness conditions on the underlying copula. For small samples, the empirical beta copula outperforms the empirical copula both in terms of bias and variance. Compared with the empirical Bernstein copula with polynomial degrees as suggested in Janssen et al.~\cite{Jan-Swan-Ver2012}, the empirical beta copula is still more accurate for several copula models, especially in terms of the bias.

The paper is organized as follows. In Section~\ref{sec:cop}, we define the empirical beta copula 
and prove its relation to the empirical Bernstein copula together with some auxiliary
results for the Bernstein transformation.  In Section~\ref{sec:asym}, we formulate and prove 
the asymptotic results for the empirical Bernstein copula process.  
A Monte Carlo simulation study is presented in Section~\ref{sec:simul}, and we conclude the paper 
with some remarks in Section~\ref{sec:concl}. Some additional proofs are given in the Appendix.

\section{The empirical beta and Bernstein copulas}
\label{sec:cop}

The sampling scheme introduced in the preceding section can be considered as making building blocks for sampling from the empirical beta copula, which will be defined in Section~\ref{subsec:beta}.  It turns out to be a particular case of the empirical Bernstein copula.  Necessary and sufficient conditions are given for the empirical Bernstein copula to be a genuine copula, conditions which imply that the empirical beta copula is a genuine copula. Finally, we provide a non-asymptotic bound for the difference between the empirical copula and the empirical beta copula.

\subsection{Empirical beta copula}
\label{subsec:beta}

We continue to use the notation given in the introduction.  To express mathematically
the idea stated above, we replace each indicator function in the definition \eqref{eq:empirical-copula}
of the empirical copula by the cumulative distribution function of the $j$th component,  
$V_{R_{i,j}^{(n)},j}^{(n)}$, of $\tilvecV_i^{(n)}$ in \eqref{eq:V-i} conditionally on $R_{i,j}^{(n)}=r$.  
Since $V_{r,j}^{(n)}$ is the $r$th order statistic of an independent random sample of size $n$
from the uniform distribution on $[0, 1]$, its distribution is a beta distribution 
$\mathcal{B}(r,n+1-r)$.  Thus we define the \emph{empirical beta copula} by
\begin{equation}
\label{eq:betacopula}
  \bbC_n^\beta(\vecu)=\frac{1}{n}\sum_{i=1}^n\prod_{j=1}^d F_{n,R_{i,j}^{(n)}}(u_j), \quad
  \vecu=(u_1,\ldots, u_d)\in [0,1]^d,
\end{equation}
where, for $u\in [0,1]$ and $r \in \{1, \ldots, n\}$,
\begin{equation}
\label{eq:betacdf}
  F_{n,r}(u) 
  = \Pr(U_{r:n}\leq u)
  = \sum_{s=r}^n \binom{n}{s} u^s (1-u)^{n-s}
\end{equation}
is the cumulative distribution function of $\mathcal{B}(r,n+1-r)$. 
Here and henceforth, $U_{1:n}<\ldots < U_{n:n}$ generically denote the order statistics 
based on $n$ independent random variables $U_1,\ldots,U_n$, uniformly distributed on $[0, 1]$.  
Since $V_{R_{i,j}^{(n)},j}^{(n)}$ has a $\mathcal{B}(r_{i,j},n+1-r_{i,j})$ distribution for 
$j\in\{1,\ldots,d\}$ and $V_{R_{i,1}^{(n)},1}^{(n)},\ldots,V_{R_{i,d}^{(n)},d}^{(n)}$ are independent, 
both conditionally on $R_{i,j}^{(n)}=r_{i,j}, \; j=1,\ldots,d$, one sees that 
picking one element randomly from among the vectors $\tilvecV_1^{(n)},\ldots,\tilvecV_n^{(n)}$ 
in \eqref{eq:V-i} amounts to sampling from the empirical beta copula $\bbC_n^\beta$ 
conditionally on $\vecX_1,\ldots,\vecX_n$.

In the absence of ties, all $d$ margins of $\bbC_n^\beta$ are equal to the uniform distribution on $[0, 1]$, so that $\bbC_n^\beta$ is a genuine copula. Indeed, for $j \in \{1, \ldots, d\}$ and $u_j \in [0, 1]$,
\begin{align*}
  \bbC_n^\beta(1,\ldots,1,u_j,1,\ldots,1)&=\frac{1}{n}\sum_{i=1}^n F_{n,R_{i,j}^{(n)}}(u_j)
  =\frac{1}{n}\sum_{r=1}^n F_{n,r}(u_j)\\
  &=\frac{1}{n}\sum_{r=1}^n \E\bigl(\indic\{U_{r:n}\leq u_j\}\bigr)
  =\E\biggl(\frac{1}{n}\sum_{r=1}^n \indic\{U_r\leq u_j\}\biggr)=u_j.
\end{align*}
In case of ties, the ranks defined in \eqref{eq:ranks} are no longer a permutation of 
$\{1, \ldots, n\}$ and the above argument breaks down.  An easy way to get around this issue
is by breaking the ties at random.

\subsection{Preliminaries on Bernstein polynomials}

Before we show that the empirical beta copula is a particular case of the empirical Bernstein copula, we need to state and prove some auxiliary results on Bernstein polynomials in general.  Put 
\[
  p_{m,s}(u) = \binom{m}{s}u^s(1-u)^{m-s},
  \qquad u\in [0,1],\ m\in\bbN,\ s\in\{0,1,\ldots,m\}.
\]
For $\vecm=(m_1,\ldots,m_d)\in\bbN^d$ and a real array $a= (a_{s_1,\ldots,s_d}\in\bbR\colon s_j=0,\ldots,m_j, \ j=1,\ldots,d)$, consider the following Bernstein polynomial:
\[
  B_{\vecm}(a)(\vecu):=\sum_{s_1=0}^{m_1}\cdots\sum_{s_d=0}^{m_d} a_{s_1,\ldots,s_d}
  \prod_{j=1}^d p_{m_j,s_j}(u_j), \quad\vecu=(u_1,\ldots,u_d)\in [0,1]^d. 
\]
Our objective in this subsection is to derive a (necessary and) sufficient condition
on the coefficients $a$ for $B_{\vecm}(a)$ to be a copula.
To this end, we need the following two lemmas.

\begin{Lem}
\label{prop:prelim}
Let $n \in \bbN$ and $a = (a_0, \ldots, a_n) \in \bbR^{n+1}$.
\begin{enumerate}
\item We have $\sum_{r=0}^na_rp_{n,r}(t) = 0$ for all $t\in [0,1]$ if and only if
$a_r=0$ for all $r \in \{0,1,\ldots,n\}$.
\item We have $\sum_{r=0}^na_rp_{n,r}(t)=t$ for all $t\in [0,1]$ if and only if 
$a_r=r/n$ for all $r \in \{0,1,\ldots,n\}$.
\end{enumerate}
\end{Lem}

\begin{proof}{Proof}
The `if' parts follow from direct computation. The `only if' part in (i) is an immediate consequence of the fact that the Bernstein polynomials $p_{n,0}, \ldots, p_{n,n}$ form a basis of the linear space of all real polynomials with degree not larger than $n$.  For the `only if' part in (ii), 
use the identity 
\begin{equation*}
  \sum_{r=0}^na_rp_{n,\,r}(t)=\sum_{k=0}^nt^k\binom{n}{k}\cdot
  \sum_{r=0}^ka_r\binom{k}{r}(-1)^{k-r}, 
\end{equation*}
and induction on $k$.  
\end{proof}

\begin{Lem}
\label{prop:bernstein-0}
$B_{\vecm}(a)(\vecu)=0$ for all $\vecu\in [0,1]^d$ if and only if $a \equiv 0$.
\end{Lem}

\begin{proof}{Proof}
The `if' part is trivial. The `only if' part can be proven by induction on $d$, using Lemma~\ref{prop:prelim}(i) both for the case $d = 1$ as for the induction step. We leave the details to the reader.
\end{proof}

For $j \in \{1, \ldots, d\}$, define the difference operator $\Delta_j$ mapping a given array $a$ to the new array $\Delta_ja$ given by
\begin{equation*}
  \Delta_ja_{s_1,\ldots,s_d} = a_{s_1,\ldots,s_j,\ldots,s_d}-a_{s_1,\ldots,s_j-1,\ldots,s_d},
  \qquad
  \left\{
    \begin{array}{l}
      s_j \in \{1, \ldots, m_j\}, \\
      s_k \in \{0, 1, \ldots, m_k\} \text{\ if\ } k \neq j.
    \end{array}
  \right.
\end{equation*}
Difference operators $\Delta_j$ with distinct indices $j$ can be composed in the obvious way. In particular,
\[
  \Delta_1\cdots\Delta_d a_{s_1,\ldots,s_d}
  = \sum_{t \in \{0, 1\}^d} (-1)^{\#\{j : t_j = 1\}} a_{s_1 - t_1,\ldots, s_d - t_d},
\]
where $s_j \in \{1, \ldots, m_j\}$ for each $j \in \{1, \ldots, d\}$.

Furthermore, consider the following three conditions on a real array $a$:
\begin{enumerate}[({C}.1)]
\setlength{\itemsep}{0pt}
\item[(C.1)] $a_{s_1,\ldots,s_d}=0$ as soon as $s_j=0$ for some $j \in \{1, \ldots, d\}$;
\item[(C.2)] $a_{m_1,\ldots,m_{j-1},s_j,m_{j+1},\ldots,m_d}=s_j/m_j$ for each $j \in \{1, \ldots, d\}$ and each $s_j \in \{0,1,\ldots,m_j\}$;
\item[(C.3)] $\Delta_1\cdots\Delta_d a_{s_1,\ldots,s_d}\geq 0$ for all $(s_1,\ldots,s_d) \in \prod_{j=1}^d \{1, \ldots, m_j\}$.
\end{enumerate}

\begin{Prop}
\label{prop:bernstein-polynomial-copula}
If the conditions (C.1), (C.2) and (C.3) hold, then the Bernstein polynomial $B_{\vecm}(a)$
is a copula.  Moreover, (C.1) and (C.2) are necessary for $B_{\vecm}(a)$ to be a copula.
\end{Prop}
Before commencing the proof, recall that a function $C\colon [0,1]^d\tendsto [0,1]$ is a copula if and only if the following conditions hold:
\begin{enumerate}[(i)]
\setlength{\itemsep}{0pt}
\item $C$ is \textit{grounded}; that is, $C(\vecu)=0$ whenever at least one of the $u_j$ equals $0$.
\item $C(1,\ldots,1,u_j,1,\ldots,1)=u_j$ for all $j\in\{1,\ldots,d\}$.
\item $C$ is \textit{$d$-increasing}; that is, for all $\vecu, \vecv\in [0,1]^d$ with
$u_j<v_j$ for all $j\in\{1,\ldots,d\}$, we have 
\[
  \sum_{\vecw} (-1)^{\#\{j\colon w_j=u_j\}}C(\vecw)\geq 0,
\]
where the sum is taken over all $\vecw$ such that $w_j \in \{u_j,v_j\}$ for all $j\in\{1,\ldots,d\}$.
\end{enumerate}
See Nelsen~\cite{Nelsen2006} or more explicitly, Theorem 1.1 of Mai and 
Scherer~\cite{Mai-Scherer2012}.  

\begin{proof}{Proof of Proposition~\ref{prop:bernstein-polynomial-copula}}
Since $p_{n,r}(0)=\indic\{r=0\}$, we have
\begin{multline*}
  B_{\vecm}(a)(u_1,\ldots,u_{j-1},0,u_{j+1},\ldots,u_d)\\
  =\sum_{s_1=0}^{m_1}\cdots\sum_{s_{j-1}=0}^{m_{j-1}}\sum_{s_{j+1}=0}^{m_{j+1}}\cdots\sum_{s_d=0}^{m_d} 
   a_{s_1,\ldots,s_{j-1},0,s_{j+1},\ldots,s_d}\prod_{k\not= j} p_{m_k,\,s_k}(u_k) 
\end{multline*}
It then follows from Lemma~\ref{prop:bernstein-0} that (C.1) is equivalent to the groundedness of $B_{\vecm}(a)$.

Since $p_{n,r}(1)=\indic\{r=n\}$, we have 
\[
  B_{\vecm}(a)(1,\ldots,1,u_{j},1,\ldots,1)
  =\sum_{s_j=0}^{m_j}a_{m_1,\ldots,m_{j-1},s_j,m_{j+1},\ldots,m_d}p_{m_j,s_j}(u_j).
\]
By Lemma~\ref{prop:prelim}(ii), condition (C.2) is equivalent to the marginal condition (ii) above for $B_{\vecm}(a)$.

Using the identity \eqref{eq:bin-derivative} and by induction, one can easily find that
\begin{align*}
  \partial_{1,\ldots, d}B_{\vecm}(a)(\vecu)
  =\sum_{s_1=1}^{m_1}\cdots\sum_{s_d=1}^{m_d}\Delta_1\cdots\Delta_d a_{s_1,\ldots,s_d}
   \prod_{j=1}^d m_j p_{m_j-1,\,s_j-1}(u_j).
\end{align*}
If (C.3) is satisfied, then it is obvious from the above expression that we have
$\partial_{1,\ldots, d}B_{\vecm}(a)(\vecu)\geq 0$ for all $\vecu\in (0,1)^d$.
Since $B_{\vecm}(a)$ is infinitely differentiable, it thus follows from 
Proposition~\ref{prop:d-increasing-partial-derivative} that $B_{\vecm}(a)$ is $d$-increasing.
Therefore (C.1), (C.2) and (C.3) imply that $B_{\vecm}(a)$ is a copula.

The necessity of (C.1) and (C.2) for $B_{\vecm}(a)$ to be a copula has already been proved.
\end{proof}

\begin{Rem}
The condition (C.3) is not necessary for $B_{\vecm}(a)$ to be a copula. 
Take $m_1=3$ and $m_2=2$, and let $a_{i,0}=a_{0,j}=0$ for $i\in\{0,1,2,3\}$, $j\in\{0,1,2\}$,
and $a_{11}=1/8$, $a_{12}=1/3$, $a_{21}=1/2$, $a_{22}=2/3$, $a_{31}=1/2$, $a_{32}=1$.
Then $(a_{ij}\colon i\in\{0,1,2,3\};\ j\in\{0,1,2\})$ satisfies (C.1) and (C.2), but not
(C.3) since $\Delta_1\Delta_2a_{22}=-1/24<0$.  By direct inspection, it is easy 
to see that $\partial_{1,2}B_{\vecm}(a)(\vecu)>0$ for all $\vecu\in [0,1]^2$
because $\partial_{1,2}B_{3,2}(a)$ is linear in $u_2$.
\end{Rem}

\subsection{Bernstein copulas}

For a function $f\colon [0,1]^d\tendsto\bbR$, the Bernstein polynomial of order $\vecm:=(m_1,\ldots,m_d) \in \bbN^d$ of $f$ is defined by
\[
  B_{\vecm}(f)(\vecu):=\sum_{s_1=0}^{m_1}\cdots\sum_{s_d=0}^{m_d} f(s_1/m_1,\ldots,s_d/m_d)
  \prod_{j=1}^d p_{m_j,s_j}(u_j). 
\]
When $C$ is a copula, $B_{\vecm}(C)$ is called the \textit{Bernstein copula} of $C$. 
For the empirical copula $\bbC_n$, we call $B_{\vecm}(\bbC_n)$ the \textit{empirical 
Bernstein copula}.

For a copula $C$, the real array $a$ defined by $a_{s_1,\ldots,s_d}=C(s_1/m_1,\ldots,s_d/m_d)$ satisfies conditions (C.1)--(C.3). By Proposition~\ref{prop:bernstein-polynomial-copula}, the Bernstein copula $B_{\vecm}(C)$ of $C$ is therefore a copula itself. On the other hand, the empirical Bernstein copula $B_{\vecm}(\bbC_n)$ is not necessarily a copula.

\begin{Prop}
\label{prop:Bernstein-genuine-copula}
Let $\bbC_n$ be the empirical copula of a sample $\vecX_1, \ldots, \vecX_n$ of $d$-variate vectors without ties in any of the $d$ components. Let $\vecm \in \bbN^d$. The empirical Bernstein copula $B_{\vecm}(\bbC_n)$ is a copula if and only if all the polynomial degrees $m_1, \ldots, m_d$ are divisors of $n$.
\end{Prop}

\begin{proof}{Proof} 
The empirical copula $\bbC_n$ in \eqref{eq:empirical-copula} is a cumulative distribution function on $[0, 1]^d$ without mass on the lower boundary of the unit hypercube. The real array $a$ defined by $a_{s_1,\ldots,s_d}=\bbC_n(s_1/m_1,\ldots,s_d/m_d)$ therefore satisfies conditions~(C.1) and~(C.3) above. By Proposition~\ref{prop:bernstein-polynomial-copula}, $B_{\vecm}(\bbC_n)$ is a copula if and only if that array $a$ also satisfies (C.2).

Let $\lfloor x \rfloor$ denote the integer part of the real number $x$. For $j \in \{1,\ldots,d\}$ and $s_j \in \{1, \ldots, m_j\}$, we have 
\begin{align*}
  \bbC_n(1,\ldots,1,s_j/m_j,1,\ldots,1)
  &= \frac{1}{n}\sum_{i=1}^n \indic\biggl\{\frac{R_{i,j}^{(n)}}{n}\leq\frac{s_j}{m_j}\biggr\} \\
  &= \frac{1}{n} \#\biggl\{i = 1, \ldots, n\colon i\leq\frac{s_jn}{m_j}\biggr\}
  = \frac{1}{n} \left\lfloor \frac{s_j n}{m_j} \right\rfloor.
\end{align*}
Condition (C.2) is fulfilled if and only if the right-hand side is equal to $s_j/m_j$ for all $s_j \in \{1, \ldots, m_j\}$. Setting $s_j = 1$ shows that this requires $m_j$ to be a divisor of $n$, and the latter condition is easily seen to be sufficient as well.
%
\end{proof}

Now we show that the empirical beta copula $\bbC_n^\beta$ in \eqref{eq:betacopula} is a particular case of the empirical Bernstein copula. 

\begin{Lem}
\label{lem:betaBernstein}
Let $\bbC_n$ be the empirical copula of a sample $\vecX_1, \ldots, \vecX_n$ of $d$-variate vectors without ties in any of the $d$ components. Then $B_{(n, \ldots, n)}( \bbC_n ) = \bbC_n^\beta$.
\end{Lem}

\begin{proof}{Proof}
The $\mathcal{B}(r,n+1-r)$ cumulative distribution function $F_{n,r}$ in \eqref{eq:betacdf} can be written in terms of Bernstein polynomials $p_{n,s}$ as
\[
  F_{n,r}(u)=\sum_{s=r}^np_{n,s}(u)=\sum_{s=0}^n\indic\{r\leq s\} \, p_{n,s}(u), \qquad u \in [0, 1].
\]
As a consequence, for $\vecu \in [0, 1]^d$,
\begin{align*}
  \bbC_n^\beta(\vecu)
  &= 
  \frac{1}{n} \sum_{i=1}^n\prod_{j=1}^d\sum_{s_j=0}^n   \indic\{R_{i,j}^{(n)}\leq s_j\} \, p_{n,s_j}(u_j) \\
  &=\frac{1}{n} \sum_{i=1}^n\sum_{s_1=0}^n\cdots\sum_{s_d=0}^n \prod_{j=1}^d\indic\{R_{i,j}^{(n)}\leq s_j\} \, p_{n,s_j}(u_j) \\
  &=\sum_{s_1=0}^n\cdots\sum_{s_d=0}^n \biggl(\frac{1}{n}\sum_{i=1}^n \prod_{j=1}^d\indic\{R_{i,j}^{(n)}\leq s_j\}\biggr)\prod_{j=1}^d p_{n,s_j}(u_j).
\end{align*}
Invoking the definition of the empirical copula $\bbC_n$ in \eqref{eq:empirical-copula} finally yields
\begin{equation*}
  \bbC_n^\beta(\vecu)=\sum_{s_1=0}^n\cdots\sum_{s_d=0}^n \bbC_n(s_1/n,\ldots,s_d/n)
  \prod_{j=1}^d p_{n,s_j}(u_j)=B_{(n,\ldots,n)}(\bbC_n)(\vecu)
\end{equation*}
as required. 
\end{proof}

Proposition~\ref{prop:Bernstein-genuine-copula} and Lemma~\ref{lem:betaBernstein} confirm that the empirical beta copula is itself a copula. We had already reached this conclusion in Subsection~\ref{subsec:beta} by a direct argument.

\begin{Rem}
Sancetta and Satchell~\cite{Sance-Satch2004} first introduced the Bernstein copula.
In that paper, it is falsely claimed (p.~537) that a function $C$ on $[0,1]^d$ is a copula 
if and only if $C$ is nondecreasing in all its arguments and satisfies the Fr\'{e}chet bounds
\[
  \max\{0,u_1+\cdots+u_d-(d-1)\}\leq C(u_1,\ldots,u_d)\leq
  \min(u_1,\ldots,u_d);
\]
for a counterexample, see Nelsen~\cite[Exercise 2.11]{Nelsen2006}.
Our results above correct both the statement and the proof of their Theorem~1.
\end{Rem}

\subsection{Proximity of the empirical copula and the empirical beta copula}

We provide a deterministic, non-asymptotic bound for the difference between the empirical copula $\bbC_n$ and the empirical beta copula $\bbC_n^\beta$.

\begin{Prop}
\label{prop:nonasymp-bd}
Let $\bbC_n$ and $\bbC_n^\beta$ be the empirical copula and the empirical beta copula, respectively, of a sample $\vecX_1, \ldots, \vecX_n$ of $d$-variate vectors without ties in any of the $d$ components. We have
\begin{equation}
\label{eq:nonasymp-bd}
  \sup_{\vecu\in [0,1]^d}|\bbC_n(\vecu)-\bbC_n^\beta(\vecu)|
  \leq d\{n^{-1/2}(\log n)^{1/2}+n^{-1/2}+n^{-1}\}.
\end{equation}
\end{Prop}

\begin{proof}{Proof}
Consider the ranks $R_{i,j}^{(n)}$ as in \eqref{eq:ranks}. Using the identity 
\[
  \prod_{i=1}^d a_i-\prod_{i=1}^d b_i=\sum_{i=1}^d (a_i-b_i)
  \prod_{j=1}^{i-1}b_j\prod_{h=i+1}^d a_h,
\]
we have, for $\vecu \in [0, 1]^d$, combining \eqref{eq:empirical-copula} and \eqref{eq:betacopula},
\begin{align*}
  \lvert \bbC_n(\vecu)-\bbC_n^\beta(\vecu) \rvert
  &\leq \frac{1}{n} \sum_{i=1}^n
  \left\lvert \prod_{j=1}^d \indic\{R_{i,j}^{(n)} / n \leq u_j\}-\prod_{j=1}^d F_{n,R_{i,j}^{(n)}}(u_j) \right\rvert\\
  &\leq\frac{1}{n} \sum_{i=1}^n\sum_{j=1}^d
  \left\lvert \indic\{R_{i,j}^{(n)} \leq nu_j\}-F_{n,R_{i,j}^{(n)}}(u_j) \right\rvert.
\end{align*}
Let $a\geq 0$ and fix $\vecu \in [0, 1]^d$.  Let $I_n(a)$ be the set of indices
$i\in\{1,\ldots,n\}$ such that $R_{i,j}^{(n)}\in [n u_j-a, n u_j+a]$ for some
$j\in\{1,\ldots,d\}$.  The set $I_n(a)$ contains at most $(2a+1)d$ elements
because for each $j\in\{1,\ldots,d\}$, the ranks $R_{1,j}^{(n)},\ldots,R_{n,j}^{(n)}$ 
constitute a permutation of $\{1,\ldots,n\}$.  It follows that
\begin{equation}
\label{eq:nonasymp-bd:1}
  \lvert \bbC_n(\vecu)-\bbC_n^\beta(\vecu) \rvert
  \leq 
  \frac{(2a+1)d}{n} +
  \frac{1}{n} \sum_{\substack{i = 1, \ldots, n\\i\notin I_n(a)}} 
  \sum_{j=1}^d \left\lvert \indic\{R_{i,j}^{(n)} \leq nu_j\} - F_{n,R_{i,j}^{(n)}}(u_j) \right\rvert.
\end{equation}

Let $r\in\{1,\ldots,n\}$ and $u\in [0,1]$. Let $B$ denote a Binomial random variable with $n$ trials and success probability $u$. If $r<nu-a$, then, in view of \eqref{eq:betacdf} and by Hoeffding's inequality, 
\[
  1-F_{n,r}(u) = \Pr(B < r) \leq \Pr(B < nu-a) \leq \exp(-2a^2/n).
\]
Similarly, if $r>nu+a$, then 
\[
  F_{n,r}(u)=\Pr(B \geq r) \leq \Pr(B > nu+a) \leq \exp(-2a^2/n).
\]
The two displays taken together imply that for $r \in \{1, \ldots, n\} \setminus [nu-a, nu+a]$, we have
\begin{equation}
\label{eq:nonasymp-bd:2}
  \lvert \indic\{r \le nu\} - F_{n,r}(u) \rvert \le \exp( - 2a^2/n ).
\end{equation}

Combine \eqref{eq:nonasymp-bd:1} and \eqref{eq:nonasymp-bd:2} to see that
\[
  \lvert \bbC_n(\vecu)-\bbC_n^\beta(\vecu)\rvert \leq \frac{(2a+1)d}{n}+d\exp(-2a^2/n).
\]
The right-hand side does not depend on $\vecu$ and we still have the liberty of choosing $a \ge 0$. Choosing $a=(1/2)\sqrt{n\log n}$ yields \eqref{eq:nonasymp-bd}.
\end{proof}

By Proposition~\ref{prop:nonasymp-bd}, the empirical beta copula is uniformly consistent whenever the empirical copula is. Nevertheless, the bound in \eqref{eq:nonasymp-bd} is too weak to conclude that the empirical processes $\sqrt{n} ( \bbC_n^\beta - C )$ and $\sqrt{n} ( \bbC_n - C )$ are asymptotically the same. This is the topic of the Section~\ref{sec:asym}.

\section{Asymptotics for the empirical beta and Bernstein copulas}
\label{sec:asym}

We provide asymptotic theory for the empirical Bernstein copula process $\sqrt{n} \{ B_{\vecm}( \bbC_n ) - C \}$ under weaker assumptions than those in Janssen et al.~\cite{Jan-Swan-Ver2012}. Here $\vecm = \vecm(n) \in \bbN^d$ is a sequence of multi-indices such that $m_* = \min(m_1, \ldots, m_d) \to \infty$ as $n \to \infty$.

For $\vecm \in \bbN^d$ and $\vecu \in [0, 1]^d$, let $\mu_{\vecm, \vecu}$ be the law of the random vector $(S_1/m_1, \ldots, S_d/m_d)$, where $S_1, \ldots, S_d$ are independent random variables, the law of $S_j$ being Binomial$(m_j, u_j)$ for each $j \in \{1, \ldots, d\}$. For asymptotic analysis, it is convenient to write the empirical Bernstein copula as a mixture:
\[
  B_{\vecm}(\bbC_n)(\vecu)=\int_{[0,1]^d}\bbC_n(\vecw)\,{\d}\mu_{\vecm,\vecu}(\vecw),
  \qquad \vecu \in [0, 1]^d.
\]

The empirical copula process and the empirical Bernstein copula process are defined by
\begin{align*}
  \bbG_n &= \sqrt{n}(\bbC_n-C), \\
  \bbG_{n,\vecm} &= \sqrt{n} \{ B_{\vecm}( \bbC_n ) - C \}.
\end{align*}
In view of Lemma~\ref{lem:betaBernstein}, the special case $\vecm = (n, \ldots, n)$ yields 
the empirical beta copula process:
\[
  \bbG_n^\beta = \sqrt{n} ( \bbC_n^\beta - C ).
\]

Noting that $B_{\vecm}(C)(\vecu)=\int_{[0,1]^d}C(\vecw)\,{\d}\mu_{\vecm,\vecu}(\vecw)$, 
we have
\begin{equation}
  \label{eq:empirical-decomp}
  \bbG_{n,\vecm}(\vecu)
  =
  \int_{[0,1]^d}\bbG_n(\vecw)\,{\d}\mu_{\vecm,\vecu}(\vecw)
  +
  \sqrt{n}\{B_{\vecm}(C)(\vecu)-C(\vecu)\}.
\end{equation}
We call the first and second term on the right-hand side of \eqref{eq:empirical-decomp} the stochastic term and the bias term, respectively. We deal with these terms in Subsections~\ref{subsec:stoch} and~\ref{subsec:bias}. The combination of both analyses in Subsection~\ref{subsec:thm} then yields our main result, Theorem~\ref{thm:empBernstein} below, on the asymptotic distribution of the empirical Bernstein copula process.

\subsection{Stochastic term}
\label{subsec:stoch}

Let $\ell^\infty([0,1]^d)$ be the Banach space of real-valued, bounded functions on $[0,1]^d$,
equipped with the supremum norm $\lVert \cdot \rVert_\infty$. The arrow $\rightsquigarrow$ denotes
weak convergence in the sense used in van der Vaart and Wellner~\cite{Vaart-Wellner}.

If $\vecX_1, \vecX_2, \ldots$ are independent and identically distributed random vectors from a common continuous distribution with copula $C$, then, provided $C$ satisfies Condition~\ref{cond:S} below, we have (Segers~\cite{Segers2012})
\[
  \bbG_n\rightsquigarrow\bbG^C, \qquad n \to \infty
\]
in $\ell^\infty([0, 1]^d)$, where
\begin{equation}
\label{eq:empCopProcLimit}
  \bbG^C(\vecu) = \bbU^C(\vecu)-\sum_{j=1}^d \dot{C}_j(\vecu) \, \bbU^C(\indic,u_j,\indic)
\end{equation}
and where $\bbU^C$ is a centered, Gaussian process with continuous trajectories and covariance function
\begin{equation}
\label{eq:empCopProcLimitCov}
  \Cov \{ \bbU^C(\vecu),\bbU^C(\vecv) \} 
  = C(\vecu\wedge\vecv)-C(\vecu) \, C(\vecv),
  \qquad \vecu, \vecv \in [0, 1]^d.
\end{equation}
More generally, if $\vecX_1, \vecX_2, \ldots$ is a strictly stationary time series whose stationary $d$-variate distribution is continuous and has copula $C$ satisfying Condition~\ref{cond:S}, then, provided certain mixing or weak dependence conditions hold, we still have weak convergence in $\ell^\infty([0, 1]^d)$ of the empirical copula process, as shown in B\"{u}cher and Volgushev~\cite{Buech-Volg2013}:
\begin{equation}
\label{eq:bbGn2bbG}
  \bbG_n \rightsquigarrow \bbG, \qquad n \to \infty,
\end{equation}
with $\bbG$ a tight, centered Gaussian process on $[0, 1]^d$ whose covariance structure depends on the full distribution of the time series.

In the analysis of the stochastic term in \eqref{eq:empirical-decomp}, the weak convergence in \eqref{eq:bbGn2bbG} is all that we need to know of the empirical copula process.

\begin{Prop}
\label{lem:stochastic-term}
If $\bbG_n\rightsquigarrow\bbG$ in $\ell^\infty([0,1]^d)$ as $n\tendsto\infty$, and if the limiting process $\bbG$ has continuous trajectories almost surely, then, whenever $m_*=m_*(n)\tendsto\infty$ as $n\tendsto\infty$,  
\[
  \sup_{\vecu\in [0,1]^d}\left|\int_{[0,1]^d}\bbG_n(\vecw)\,{\d}\mu_{\vecm,\vecu}(\vecw)
  -\bbG_n(\vecu)\right|=\oP(1). 
\]
\end{Prop}


\begin{proof}{Proof}
Let $|\cdot |_\infty$ denote the maximum norm on $\bbR^d$.
For $\epsilon>0$, we have 
\begin{multline*}
  \int_{[0,1]^d} \lvert \bbG_n(\vecw)-\bbG_n(\vecu) \rvert \,{\d}\mu_{\vecm,\vecu}(\vecw)\\
  \leq 2\, \lVert\bbG_n\rVert_\infty \, \mu_{\vecm,\vecu}(\{\vecw\colon \lvert \vecw-\vecu \rvert_\infty>\epsilon\})
  +\sup_{ \lvert \vecx-\vecy \rvert_\infty\leq\epsilon} \lvert \bbG_n(\vecx)-\bbG_n(\vecy) \rvert.
\end{multline*}
By the assumption, $\bbG_n$ is stochastically equicontinuous.  Hence, for a given $\eta>0$, 
we can find $\epsilon=\epsilon(\eta)>0$ sufficiently small such that
\[
  \limsup_{n\tendsto\infty}\Pr\biggl\{\sup_{|\vecx-\vecy|_\infty\leq\epsilon}|\bbG_n(\vecx)-\bbG_n(\vecy)|
  >\eta\biggr\}\leq\eta.
\]
Moreover, the weak convergence $\bbG_n\rightsquigarrow\bbG$ implies that
$\|\bbG_n\|_\infty=\OP(1)$ as $n\tendsto\infty$.  Finally, Chebyshev's inequality yields 
\[
  \mu_{\vecm,\vecu}(\{\vecw\colon \lvert\vecw-\vecu\rvert_\infty>\epsilon\})
  \leq\sum_{j=1}^d\Pr( \lvert S_j/m_j-u_j \rvert > \epsilon)\leq\frac{d}{4\epsilon^2m_*},
\]
which converges to zero as $m_*\tendsto\infty$, uniformly in $\vecu\in [0,1]^d$.  
Since $\eta$ was arbitrary, the conclusion follows.
\end{proof}

\subsection{Bias term}
\label{subsec:bias}

Now we turn to the bias term on the right-hand side of \eqref{eq:empirical-decomp}:
\[
  \sqrt{n}\{B_{\vecm}(C)(\vecu)-C(\vecu)\}=
  \sqrt{n}\int_{[0,1]^d} \{ C(\vecw)-C(\vecu) \} \,{\d}\mu_{\vecm,\vecu}(\vecw).
\]
One-dimensional Bernstein polynomials are well studied, and results on the accuracy
of the approximation of a function by its associated Bernstein polynomial are presented in Lorentz~\cite{Lorentz86} and DeVore and 
Lorentz~\cite[Chapter 10]{DeVore-Lorentz93}.  We need to extend these to the multidimensional case.

\begin{Lem}
\label{lem:bias-term1}
For a $d$-variate copula $C$ and a multi-index $\vecm \in \bbN^d$ we have, writing $m_* = \min(m_1, \ldots, m_d)$,
\[
  \lVert B_{\vecm}(C) - C \rVert_\infty
  \leq
  \frac{d}{2\sqrt{m_*}}.
\]
\end{Lem}

\begin{proof}{Proof}
Any copula is Lipschitz with respect to the $L_1$-norm with Lipschitz constant equal to $1$. We obtain
\[
  \lVert B_{\vecm}(C) - C \rVert_\infty
  \leq
  \int_{[0,1]^d} \lvert C(\vecw)-C(\vecu) \rvert \,{\d}\mu_{\vecm,\vecu}(\vecw)
  \leq\sum_{j=1}^d\int_{[0,1]^d} \lvert w_j-u_j \rvert \,{\d}\mu_{\vecm,\vecu}(\vecw).
\]
Fix $j \in \{1, \ldots, d\}$ and let $S_j$ be a Binomial$(m_j, u_j)$ random variable. By the Cauchy--Schwarz inequality,
\begin{equation}
\label{ineq:binomial-bound}
  \int_{[0,1]^d} \lvert w_j-u_j \rvert \,{\d}\mu_{\vecm,\vecu}(\vecw) 
  \leq \E ( \lvert S_j/m_j - u_j \rvert )
  \le \sqrt{ \Var(S_j/m_j) }
  \le \frac{1}{2\sqrt{m_j}},
\end{equation}
which yields the conclusion.
\end{proof}

According to Lemma~\ref{lem:bias-term1}, the uniform convergence rate of $B_{\vecm}(C)$
to $C$ is at least $O(1/\sqrt{m_*})$ as $m_*\tendsto\infty$.  With a bit of additional 
smoothness, this can be strengthened to $o(1/\sqrt{m_*})$.  The following condition
originates from Segers~\cite{Segers2012}.

\begin{Cond}
\label{cond:S}
For each $j \in \{1,\ldots,d\}$, the copula $C$ has a continuous 
first-order partial derivative $\dot{C}_j(\vecu)=\partial C(\vecu)/\partial u_j$ 
on the set $I_j = \{\vecu\in [0,1]^d\colon 0<u_j<1\}$.  
\end{Cond}


\begin{Prop}
\label{lem:bias-term2}
If the copula $C$ satisfies Condition~\ref{cond:S}, then
\[
  \lim_{m_*\tendsto\infty}\sqrt{m_*} \lVert B_{\vecm}(C) - C \rVert_\infty = 0.
\]
\end{Prop}

\begin{proof}{Proof}
By monotonicity and Lipschitz continuity of $C$, we have
$0\leq \dot{C}_j\leq 1$.  Fix $\vecu,\,\vecw\in [0,1]^d$ and write
$\vecw(t)=\vecu+t(\vecw-\vecu)$ for $t\in [0,1]$.  The function 
$t\mapsto C(\vecw(t))$ is continuous and, by Condition~\ref{cond:S}, continuously differentiable 
on $(0,1)$ with derivative 
\[
  \frac{\d}{\d t}C(\vecw(t))=\sum_{j=1}^d(w_j-u_j) \, \dot{C}_j(\vecw(t)), \qquad t\in (0,1).
\]
For each $j$ such that $w_j=u_j$, the corresponding term in the sum is zero
no matter how $\dot{C}_j(\vecw(t))$ is defined.  If $w_j\not= u_j$, then $0<u_j+t(w_j-u_j)<1$ 
whenever $0<t<1$, so that the partial derivative $\dot{C}_j$ is well defined at $\vecw(t)$ by Condition~\ref{cond:S}.  
By the fundamental theorem of calculus, we get
\[
  C(\vecw)-C(\vecu)=\sum_{j=1}^d(w_j-u_j)\int_0^1 \dot{C}_j(\vecw(t))\,{\d}t.
\]
Hence, by Fubini's theorem,
\begin{align*}
  B_{\vecm}(C)(\vecu) - C(\vecu)
  &= \int_{[0, 1]^d} \{ C( \vecw ) - C( \vecu ) \} \, {\d} \mu_{\vecm, \vecu}( \vecw ) \\
  &= \sum_{j=1}^d\int_0^1
  \left\{ \int_{[0,1]^d}(w_j-u_j)\,\dot{C}_j(\vecw(t))\,{\d}\mu_{\vecm,\vecu}(\vecw)\right\}\,{\d}t.
\end{align*}
We need to bound its absolute value uniformly in $\vecu\in [0,1]^d$.  
Fix $j\in\{1,\ldots,d\}$ and $t\in [0,1]$.  Choose $0<\delta\leq 1/2$. 

First, let $\vecu\in [0,1]^d$ be such that $u_j\in [0,\delta)\cup (1-\delta,1]$.
Since $0\leq\dot{C}_j\leq 1$, we find
\[
  \left\lvert \int_{[0,1]^d}(w_j-u_j) \, \dot{C}_j(\vecw(t))\,{\d}\mu_{\vecm,\vecu}(\vecw)\right\rvert
  \leq \int_{[0,1]^d} \lvert w_j-u_j \rvert \,{\d}\mu_{\vecm,\vecu}(\vecw)
  \leq \sqrt{\frac{u_j(1-u_j)}{m_j}}
  < \sqrt{\frac{\delta}{m_*}}.
\]

Second, let $\vecu\in [0,1]^d$ be such that $u_j\in [\delta, 1-\delta]$.
Then $\dot{C}_j(\vecu)$ is well defined.  Since $\int_{[0,1]^d}(w_j-u_j)\,{\d}\mu_{\vecm,\vecu}
(\vecw)=0$, we obtain
\[
  B_{\vecm}(C)(\vecu)-C(\vecu)=\sum_{j=1}^d \int_0^1 \left[
  \int_{[0,1]^d}(w_j-u_j) \bigl\{ \dot{C}_j(\vecw(t))-\dot{C}_j(\vecu) \bigr\} \,{\d}\mu_{\vecm,\vecu}(\vecw) \right] {\d}t.
\]
Choose $0<\epsilon\leq\delta/2$.  Let $|\cdot |_\infty$ denote the maximum norm
on $\bbR^d$.  Split the integral over $\vecw\in [0,1]^d$ into two parts,
according to whether $|\vecw-\vecu|_\infty$ is larger than $\epsilon$ or not.  
We have 
\begin{multline*}
  \left\lvert \int_{[0,1]^d}(w_j-u_j) \bigl\{ \dot{C}_j(\vecw(t))-\dot{C}_j(\vecu) \bigr\}
    \,{\d}\mu_{\vecm,\vecu}(\vecw)\right\rvert \\
  \leq
  \sup_{ \lvert \vecw-\vecu \rvert_\infty\leq\epsilon} \lvert \dot{C}_j(\vecw(t))-\dot{C}_j(\vecu) \rvert
  \int_{[0,1]^d} \lvert w_j - u_j \rvert \, {\d} \mu_{\vecm, \vecu}(\vecw) \\
  +
  \int_{[0,1]^d} \indic\{\lvert\vecw-\vecu\rvert_\infty >\epsilon\} \, \lvert w_j-u_j \rvert \,{\d}\mu_{\vecm,\vecu}(\vecw).
\end{multline*}
Here we have used the inequality $0 \le \dot{C}_j \le 1$ again. We now treat both terms on the right-hand side separately.
\begin{itemize}
\item
First, fix $\eta>0$.  By Condition~\ref{cond:S}, we can choose $\epsilon=\epsilon(\delta,\eta)>0$
sufficiently small such that, for each $j\in\{1,\ldots,d\}$, 
\[
  \sup_{\substack{\vecu\in [0,1]^d\\ \delta\leq u_j\leq 1-\delta}}
  \sup_{\lvert \vecw-\vecu\rvert_\infty\leq\epsilon}
  \lvert \dot{C}_j(\vecw)-\dot{C}_j(\vecu) \rvert
  \leq 
  \eta.
\]
Furthermore, we have $\int_{[0, 1]^d} \lvert w_j - u_j \rvert \, {\d} \mu_{\vecm, \vecu}(\vecw) \le 1 / (2\sqrt{m_*})$ by \eqref{ineq:binomial-bound}.
\item
Second, by the Cauchy--Schwarz inequality,
\[
  \int_{[0,1]^d}
    \indic\{ \lvert \vecw-\vecu \rvert_\infty > \epsilon\} \, \lvert w_j-u_j \rvert \,
  {\d}\mu_{\vecm,\vecu}(\vecw)
  \leq
  \sqrt{
    \frac%
      {\mu_{\vecm,\vecu}(\{\vecw\colon \lvert \vecw - \vecu \rvert_\infty >\epsilon\})}%
      {m_*}
  }.
\]
By the Markov inequality and \eqref{ineq:binomial-bound}, the numerator of the right-hand side
converges to zero as $m_*\tendsto\infty$, uniformly in $\vecu\in [0,1]^d$.  
\end{itemize}

By the above bounds, it follows that $\limsup_{m_*\tendsto\infty}
\sqrt{m_*}\rVert B_{\vecm}(C)-C \rVert_\infty\leq\sqrt{\delta}+\eta/2$.  Since $\delta >0$ and
$\eta>0$ were arbitrary, the result follows.
\end{proof}

Condition~\ref{cond:S} on $C$ is significantly weaker than the assumption that $C$ has bounded third-order partial derivatives, which Janssen et al.~(2012) require to prove their asymptotic results for the empirical Bernstein copula. The latter assumption is violated by many common parametric copula families. Janssen et al.~(2012) in fact proved a pointwise $O(m_*^{-1})$ convergence rate of the bias term under the assumption that $C$ has bounded third-order partial derivatives, but to obtain such a rate, the assumption is still too strong. In fact, a Lipschitz condition on the first-order partial derivatives already suffices.

\begin{Prop}
\label{lem:bias-term3}
Let $C$ be a copula satisfying Condition~\ref{cond:S}. Let $\vecu \in (0, 1]^d$ and suppose that there exists $\eps > 0$ such that, for each $j \in \{1, \ldots, d\}$ such that $u_j < 1$, we have $0 < u_j - \eps < u_j + \eps < 1$ and $\dot{C}_j$ is Lipschitz on $\{ \vecw \in [0, 1]^d : \lvert \vecw - \vecu \rvert_\infty \le \eps \}$. Then
\[
  \sup_{\vecm} m_* \lvert B_{\vecm}(C)(\vecu) - C(\vecu) \rvert < \infty.
\]
\end{Prop}

\begin{proof}{Proof}
If $u_j = 1$ for some $j \in \{1, \ldots, d\}$, then, since the Binomial$(m_j, u_j = 1)$ distribution is concentrated at $m_j$, we can omit the $j$th coordinate altogether and pass to the appropriate $(d-1)$-dimensional margins of $B_{\vecm}(C)$ and $C$. In this way, we can eliminate all margins $j$ such that $u_j = 1$. It follows that, without loss of generality, we can assume that $\vecu \in (0, 1)^d$.

As in the proof of Proposition~\ref{lem:bias-term2}, we have the representation
\[
  B_{\vecm}(C)(\vecu) - C(\vecu)
  =
  \sum_{j=1}^d 
  \int_0^1 
    \left[
      \int_{[0,1]^d}
        (w_j-u_j) \,
	\bigl\{ \dot{C}_j(\vecw(t)) - \dot{C}_j(\vecu) \bigr\} \,
      {\d}\mu_{\vecm,\vecu}(\vecw) 
    \right] 
  {\d}t.
\]
We fix $t \in (0, 1)$ and $j \in \{1, \ldots, d\}$ and split the integral over $\vecw$ into two parts, according to whether $\lvert \vecw - \vecu \rvert_\infty$ is larger than $\eps$ or not. This yields the following bound:
\begin{multline}
\label{eq:integral_decomp}
  m_*
  \left\lvert
    \int_{[0,1]^d}
      (w_j-u_j) \,
      \bigl\{ \dot{C}_j(\vecw(t)) - \dot{C}_j(\vecu) \bigr\} \,
    {\d}\mu_{\vecm,\vecu}(\vecw) 
  \right\rvert
  \\
  \le
  \int_{[0, 1]^d}
    m_* \, 
    \lvert w_j - u_j \rvert \,
    \lvert \dot{C}_j(\vecw(t)) - \dot{C}_j(\vecu) \rvert \,
    \indic \{ \lvert \vecw - \vecu \rvert_\infty \le \eps \} \,
  {\d}\mu_{\vecm,\vecu}(\vecw) \\
  +
  \int_{[0, 1]^d}
    m_* \,
    \lvert w_j - u_j \rvert \,
    \indic \{ \lvert \vecw - \vecu \rvert_\infty > \eps \}
  {\d}\mu_{\vecm,\vecu}(\vecw).
\end{multline}
For the second term on the right-hand side, we used the fact that $0 \le \dot{C}_j \le 1$. 
We now bound the two integrals separately.

\begin{itemize}
\item
Recall that $\vecw(t) = \vecu + t ( \vecw - \vecu )$, a convex combination of $\vecu$ and $\vecw$.
The Lipschitz assumption on $\dot{C}_j$ implies that there exists a constant $L > 0$ such that
\[
  \lvert \dot{C}_j(\vecw(t)) - \dot{C}_j(\vecu) \rvert
  \le tL \sum_{k=1}^d \lvert w_k - u_k \rvert.
\]
We find that the first term on the right-hand side of \eqref{eq:integral_decomp} is bounded by
\[
  tL
  \sum_{k=1}^d
  \int_{[0, 1]^d}
    m_* \, 
    \lvert w_j - u_j \rvert \,
    \lvert w_k - u_k \rvert \,
  {\d}\mu_{\vecm,\vecu}(\vecw)
  \le
  \frac{1}{4} tLd.
\]
Here we applied the Cauchy--Schwarz inequality, by which
\begin{align*}
  \int_{[0, 1]^d}
    m_* \, 
    \lvert w_j - u_j \rvert \,
    \lvert w_k - u_k \rvert \,
  {\d}\mu_{\vecm,\vecu}(\vecw) 
  &=
  \E ( m^* \lvert S_j/m_j - u_j \rvert \, \lvert S_k/m_k - u_k \rvert ) \\
  &\le
  m^* \left[ \E\{ (S_j/m_j - u_j)^2 \} \E\{ (S_k/m_k - u_k)^2 \} \right]^{1/2} \\
  &\le
  m^* \left\{ m_j^{-1} u_j(1-u_j) \, m_k^{-1} u_k(1-u_k) \right\}^{1/2} \\
  &\le
  \left\{ u_j(1-u_j) \, u_k(1-u_k) \right\}^{1/2} \le \frac{1}{4}.
\end{align*}

\item
For the second term on the right-hand side of \eqref{eq:integral_decomp}, we can bound the indicator function in the integrand by
\[
  \indic \{ \lvert \vecw - \vecu \rvert_\infty > \eps \}
  \le
  \frac{1}{\eps} \lvert \vecw - \vecu \rvert_\infty
  \le
  \frac{1}{\eps} \sum_{k=1}^d \lvert w_k - u_k \rvert.
\]
The resulting integrals have already been bounded above in the previous term.
\end{itemize}
This finishes the proof of Proposition~\ref{lem:bias-term3}.
\end{proof}

\subsection{Empirical Bernstein and beta copula processes}
\label{subsec:thm}

By Lemma~\ref{lem:betaBernstein}, the empirical beta copula $\bbC_n^\beta$ is equal to the empirical Bernstein copula $B_{\vecm}(\bbC_n)$ with $m_j=n$ for all $j \in \{1, \ldots, d\}$. Propositions~\ref{lem:stochastic-term}, \ref{lem:bias-term2} and \ref{lem:bias-term3} then immediately imply the following theorem and corollary. Recall the empirical copula process $\bbG_n = \sqrt{n} ( \bbC_n - C ) $, the empirical Bernstein copula process $\bbG_{n, \vecm} = \sqrt{n} \{ B_{\vecm}( \bbC_n ) - C \}$, and the empirical beta copula process $\bbG_n^\beta = \sqrt{n} ( \bbC_n^\beta - C )$. 

In Theorem~\ref{thm:empBernstein}, we assume weak convergence of $\bbG_n$ to some limit process $\bbG$. See the discussion before Proposition~\ref{lem:stochastic-term} for a justification of this condition and for the form of the limit process. In particular, the condition is automatically satisfied in the iid case when $C$ satisfies Condition~\ref{cond:S}, and then the limit process is given by $\bbG^C$ in \eqref{eq:empCopProcLimit}.

\begin{Thm}
\label{thm:empBernstein}
Suppose that $C$ satisfies Condition~\ref{cond:S} and that $\bbG_n\rightsquigarrow\bbG$ in $\ell^\infty([0,1]^d)$ as $n\tendsto\infty$, the limiting process $\bbG$ having continuous trajectories almost surely. Let $\vecm = \vecm(n)$ be multi-indices such that $m_* = \min(m_1, \ldots, m_d) \to \infty$ as $n \to \infty$.

\begin{enumerate}[(i)]
\item
If $\sqrt{n} = o(m_*)$, then, for all $\vecu \in (0, 1]^d$ that satisfy the condition in Proposition~\ref{lem:bias-term3},
we have
\[
  \bbG_{n, \vecm}(\vecu) = \bbG_n( \vecu ) + o_P(1), \qquad n \to \infty.
\]

\item
If $\liminf_{n \to \infty} m_*/n > 0$, then, in $\ell^\infty([0, 1]^d)$,
\[
  \bbG_{n, \vecm} = \bbG_n + o_P(1) \rightsquigarrow \bbG, \qquad n \to \infty.
\]
\end{enumerate}
\end{Thm}

\begin{proof}{Proof}
Recall the decomposition of $\bbG_{n, \vecm}(\vecu)$ in \eqref{eq:empirical-decomp}. The stochastic term is equal to $\bbG_n(\vecu) + o_P(1)$, uniformly in $\vecu \in [0, 1]^d$, by Proposition~\ref{lem:stochastic-term}.

(i) Since $\limsup_{n \to \infty} \sqrt{n} / m_* = 0$, the bias term converges to zero at the given point $\vecu$.

(ii) Since $\limsup_{n \to \infty} \sqrt{n} / \sqrt{m_*} < \infty$, the bias term converges to zero uniformly in $\vecu \in [0, 1]^d$ thanks to Proposition~\ref{lem:bias-term2}.
\end{proof}

Setting $\vecm(n) = (n, \ldots, n)$ in Theorem~\ref{thm:empBernstein}(ii) yields the following corollary.

\begin{Cor}
\label{cor:empBeta}
Suppose that $C$ satisfies Condition~\ref{cond:S} and that $\bbG_n\rightsquigarrow\bbG$ in $\ell^\infty([0,1]^d)$ as $n\tendsto\infty$, the limiting process $\bbG$ having continuous trajectories almost surely. Then, in $\ell^\infty([0, 1]^d)$,
\[
  \bbG_n^\beta = \bbG_n + o_P(1) \rightsquigarrow \bbG, \qquad n \to \infty.
\]
\end{Cor}

The conclusion of this section is that under reasonable conditions, the empirical Bernstein and empirical beta copulas have the same large-sample distribution as the empirical copula. The real advantage of the use of the smoothed versions is visible mainly for small samples. The difference can perhaps be quantified via higher-order asymptotic theory. Instead, we assess the finite-sample performance through Monte Carlo simulations.

\section{Finite-sample performance}
\label{sec:simul}

In this section, we compare the finite-sample performance of the empirical beta copula with those of various other estimators by a Monte Carlo experiment. The simulations were performed in \textsf{R} \cite{R} using the package \textsf{copula} \cite{copulaR}.

We show the results for five copulas, three bivariate ones and two trivariate ones:
\begin{itemize}
\setlength{\itemsep}{0pt}
\item the bivariate Farlie--Gumbel--Morgenstern (FGM) copula with negative dependence ($\theta = -1$);
\item the bivariate independence copula;
\item the bivariate Gaussian copula with positive dependence ($\rho = 0.5$);
\item the trivariate $t$ copula with degrees-of-freedom parameter equal to $\nu = 4$ and pairwise correlation parameters equal to $\rho_{12} = -0.2$, $\rho_{13} = 0.5$ and $\rho_{23} = 0.4$;
\item the trivariate nested Archimedean copula with Frank generators and with Kendall's tau equal to $0.3$ at the upper node and $0.6$ at the lower node.
\end{itemize}
See, e.g., Joe~\cite{Joe1997} and Nelsen~\cite{Nelsen2006} for the explicit functional forms and properties of these copulas. We have run simulations for many other copula models, but the results were qualitatively the same as the ones showed here.

We compared the following five estimators:
\begin{itemize}
\item the empirical copula $\bbC_n$ in \eqref{eq:empirical-copula};
\item the empirical checkerboard copula $\bbC_n^\sharp$ in \eqref{eq:emp-checker-copula} below;
\item the empirical beta copula $\bbC_n^\beta$ in \eqref{eq:betacopula};
\item the empirical Bernstein copula $B_{\bm{m}}( \bbC_n )$ with $\bm{m} = (m, \ldots, m)$ and $m = \lceil n/3 \rceil$, the smallest integer not smaller than $n/3$;
\item the empirical Bernstein copula $B_{\bm{m}}( \bbC_n )$ with $\bm{m} = (m, \ldots, m)$ and $m$ as recommended by Janssen et al.\ \cite{Jan-Swan-Ver2012}, see \eqref{eq:JSV} below.
\end{itemize}

For each possible combination of estimator and model and for sample sizes $n$ from $20$ up to $100$, we computed the integrated (over the unit cube) squared bias, the integrated variance, and the integrated mean squared error by Monte Carlo simulation. See Appendix~\ref{app:estim} for the definitions of the three performance measures and a description of our computation method.

\subsection{Comparison with the empirical copula and its variant}
\label{subsec:comparison-emp}

As the simplest version of smoothed empirical copula, we introduce the \textit{empirical checkerboard copula} defined by 
\begin{equation}
  \label{eq:emp-checker-copula}
  \mathbb{C}_n^{\#}(\bm{u}) := \frac{1}{n} \sum_{i=1}^n \prod_{j=1}^d 
  \min \{ \max ( nu_j  - R_{i,j}^{(n)} + 1, 0 ), 1 \}
\end{equation}
(Li et al.~\cite{LMST97} for $d=2$, and Carley and Taylor~\cite{Carley-Taylor2002} for general $d$; see also Genest and Ne\v{s}lehov\'{a}~\cite{Gen-Nes2007} and Genest et al.~\cite{Gen-Nes-Rem2014}). We note that this is a genuine copula, just like the empirical beta copula. 

In each of the graphs in Figure~\ref{fig:emp-checker-beta}, the horizontal axis indicates the sample size; the vertical axis indicates the integrated squared bias, the integrated variance and the integrated mean squared error for the left, middle and right panels, respectively (the same comment also applies to the other figures).

From the figures, one sees that the empirical copula performs worst with respect to all three measures.  The empirical beta copula and the empirical checkerboard copula are comparable in terms of bias, but the variance of the former is smaller than the one of the latter, resulting in a smaller mean squared error for the empirical beta copula in all cases considered. Note that the empirical checkerboard copula can be thought of as smoothing at bandwidth $O(1/n)$, whereas the empirical beta copula involves smoothing at bandwidth $O(1/\sqrt{n})$. 

It is also interesting to see on which parts of the unit cube the empirical beta copula 
performs better than the empirical copula.  The localized relative efficiency on a set $B$ 
of the empirical copula $\bbC_n$ with respect to the empirical beta copula $\bbC_n^\beta$ is 
defined as
\begin{equation}
\label{eq:LRE}
  \operatorname{LRE}( \bbC_n, \bbC_n^\beta; B )
  =
  100 \% \times \frac{\operatorname{LIMSE}( \bbC_n^\beta; B )}{\operatorname{LIMSE}( \bbC_n; B )},
\end{equation}
where $\operatorname{LIMSE}( \widehat{C}_n; B )$, the localized integrated mean squared error of an estimator $\widehat{C}_n$ on $B$, is defined in Appendix~\ref{app:estim}; it is basically the mean squared error averaged over $B$.  Thus, the smaller the value of $\operatorname{LRE}(\bbC_n,\bbC_n^\beta;B)$, the better the performance of the empirical beta copula in comparison to the empirical copula on $B$.

Figure~\ref{fig:LRE} shows heat maps of the above localized relative efficiencies in case $d = 2$ for square cells $B$ of the form $[(j-1)/10, j/10] \times [(k-1)/10, k/10]$, $(j,k) \in \{1,\ldots,10\}^2$, at sample size $n = 100$ and when the true copula $C$ is equal to the independence copula, the Farlie--Gumbel--Morgenstern copula with $\theta = -1$, the Gaussian copula with correlation parameter $\tau = 0.5$, and the Gumbel copula with Kendall's tau equal to $\tau = 0.5$.
The results vary from copula to copula, but on the whole one could argue that 
the efficiency improvements are largest near the upper and right borders of the unit square.  
We note that it is on these parts of the unit square that the effect of $\bbC_n^\beta$ being
a genuine copula should be prominently visible.

\subsection{Comparison with the empirical Bernstein copulas}
\label{subsec:comparison-emp-Bernstein}

In this subsection, we consider the empirical Bernstein copula with all the orders $m_1,\ldots,m_d$ equal to an integer $m$.  Our first choice for the order of Bernstein polynomial is $m = \lceil n/3 \rceil$, where $\lceil x \rceil$ denotes the smallest integer not smaller than $x$.  With this choice, we see the consequence of Proposition~\ref{prop:Bernstein-genuine-copula} because $m$ is a divisor of $n$ only when $n$ is a multiple of 3.  

Our second choice for $m$ is based on Janssen et al.~\cite{Jan-Swan-Ver2012}, who recommend, in the bivariate case, the following choice for $m$:
\begin{equation}
\label{eq:JSV}
  m_0(u_1, u_2) = \left\{ \frac{4 b^2(u_1, u_2)}{V(u_1, u_2)} \right\}^{2/3} n^{2/3},
\end{equation}
(or the integer part thereof), where
\begin{align*}
  b(u_1, u_2) &= 
  \frac{1}{2} \sum_{j=1}^2 u_j(1-u_j) \, \ddot{C}_{jj}(u_1, u_2), \\
  V(u_1, u_2) &= 
  \sum_{j=1}^2 \dot{C}_j(u_1, u_2) \bigl\{ 1 - \dot{C}_j(u_1, u_2) \bigr\} \left\{ \frac{u_j(1-u_j)}{\pi} \right\}^{1/2}
\end{align*}
and with $\dot{C}_j$ and $\ddot{C}_{jj}$ the first- and second-order partial derivatives of $C$ with respect to $u_j$, for $j \in \{1, 2\}$.  The following remarks on this choice $m_0$ are in order.
\begin{itemize}
\item This choice requires the knowledge of $C$ through the derivatives $\dot{C}_j$ and $\ddot{C}_{jj}$. To estimate these derivatives in small samples is not easy.
\item For the independence copula, we have $\ddot{C}_{jj} \equiv 0$, so that it is not clear how to define $m_0$.
\item In Lemma~3(i) in \cite{Jan-Swan-Ver2012}, an $o_P(1)$ term is written but subsequently neglected. However, the order of magnitude of this term could be larger than the remainder terms that are analyzed later on.  
\end{itemize}
In the simulations, we take the true (but in practice unknown) value for $m_0(u_1, u_2)$ for the bivariate FGM and Gaussian copulas.

From Figure~\ref{fig:beta-bern3-bernJSV}, it is observed that the performance of the empirical Bernstein copulas are strongly affected by the periodicity due to the divisibility relation between $n$ and $m$, especially in terms of bias.  Namely, when $n$ is a multiple of 3, the integrated bias and the integrated mean squared error of the empirical Bernstein copula $B_{\bm{m}}( \bbC_n )$ with $m = \lceil n/3 \rceil$ are much smaller than the other cases.  These numerical results strongly suggest that this is because $B_{\bm{m}}( \bbC_n )$ with $m = \lceil n/3 \rceil$ is a genuine copula only when $n$ is a multiple of 3.  This might be partially verified by the simulation results in Section 4.1 although the target for comparison is the empirical copula rather than the empirical Bernstein copula; What they have in common is the larger integrated bias and integrated mean squared error in the case where they are not a copula.

Compared to the empirical Bernstein copulas, the empirical beta copula has a smaller bias and a larger variance in every case.  There is no clear ordering of the three in terms of the mean squared error.  It is interesting to note that while the mean squared error of the empirical beta copula is smaller in positively dependent cases, it is larger in negatively dependent cases. At this point, we have no explanation for this phenomenon.

\section{Concluding Remarks}
\label{sec:concl}

The empirical beta copula can be considered as a simple but effective way of correcting and smoothing the empirical copula. No smoothing parameter needs to be chosen.  The empirical beta copula is a special case of the empirical Bernstein copula, but in contrast to the latter, it is always a genuine copula. Moreover, it is extremely simple to simulate samples from it. The asymptotic distribution of the empirical beta copula is the same as that of the empirical copula, but in small samples, it performs better both in terms of bias and variance.  Moreover, there seems little to be gained from using Bernstein smoothers at other polynomial degrees $m$, except in special cases such as the independence copula. 

We also note that the only properties of the kernels $\mu_{\vecm, \vecu}$ that we needed in the asymptotic analysis were moment properties and concentration inequalities of their margins around the points $u_j$. 
This implies that the approach may perhaps be extended to products of other families of smoothing kernels than beta kernels.

Our smoothing procedure might also have a beneficial effect on the accuracy of resampling schemes for the empirical copula process (B\"{u}cher and Dette~\cite{Buech-Dette2010}).  More specifically, testing procedures based on the empirical copula typically rely on the bootstrap for the computation of the critical values of the test statistic.  For finite samples, the accuracy is often not very good: the true type I error of the test may differ greatly from the nominal one.  Then, the question is how to construct a bootstrap or multiplier resampling scheme for the empirical beta and Bernstein copulas.  This is a topic for future research, together with the question of higher-order asymptotics for the various nonparametric copula estimators.

\subsection*{Acknowledgements}

J. Segers gratefully acknowledges funding by contract ``Projet d'Act\-ions de Re\-cher\-che Concert\'ees'' No.\ 12/17-045 of the ``Communaut\'e fran\c{c}aise de Belgique'', by IAP research network Grant P7/06 of the Belgian government (Belgian Science Policy), and by the ``Projet de
Recherche'' No.\ FRFC PDR T.0146.14 of the ``Fonds de la Recherche Scientifique -- FNRS'' (Belgium).
H. Tsukahara is supported by JSPS KAKENHI Grant Number 15H03337.  The authors are also grateful
to two anonymous reviewers for their careful reading of the manuscript and their helpful
comments.

\appendix

\section{Appendix}

We derive a convenient expression for the derivative of the Bernstein-type polynomial (making no claim of originality). 
Put $p'_{n,\,r}(t):=(\d/\d t)p_{n,\,r}(t)$.  Then we have
\begin{align*}
  p'_{n,\,r}(t)&=\binom{n}{r}\bigl\{rt^{r-1}(1-t)^{n-r}-(n-r)t^r(1-t)^{n-r-1}\bigr\}\\[5pt]
              &=\begin{cases}
                -n(1-t)^{n-1} & \text{if $r=0$}\\[5pt]
                \displaystyle\frac{n!}{(r-1)!(n-r)!}t^{r-1}(1-t)^{n-r}
                -\frac{n!}{r!(n-r-1)!}t^r(1-t)^{n-r-1} & \text{if $r=1,\ldots,n-1$}\\[10pt]
                nt^{n-1} & \text{if $r=n$}
                \end{cases}
\end{align*}
Using the above expression and summation by parts, it is easy to see that 
for any sequence $a_r$, $r=0,\ldots,n$, the following identity holds:
\begin{equation}
  \label{eq:bin-derivative}
  \sum_{r=0}^na_rp'_{n,\,r}(t)=\sum_{r=1}^{n}(a_{r}-a_{r-1})np_{n-1,\,r-1}(t)
\end{equation}


When a function is differentiable, we can give a condition for $d$-increasing property 
in terms of its partial derivative.

\begin{Prop}
\label{prop:d-increasing-partial-derivative}
Suppose that a function $f\colon [0,1]^d\tendsto [0,1]$ is infinitely differentiable 
on $(0,1)^d$.  Then $f$ is $d$-increasing if and only if 
$\partial_{1,\ldots,d}f(\vecu)\geq 0$ for all $\vecu\in (0,1)^d$.  
\end{Prop}

\bigskip\noindent
\textit{Proof}. Suppose that $f$ is $d$-increasing.  Then, since $f$ is infinitely 
differentiable on $(0,1)^d$, it follows that for $\vecu\in (0,1)^d$
and $h_j>0$, $j\in\{1,\ldots,d\}$, 
\[
  \partial_{1,\ldots,d}f(\vecu)=\lim_{\substack{h_j\tendsto +0\\ j=1,\ldots,d}}
  \biggl(\prod_{j=1}^d h_j\biggr)^{-1}\sum (-1)^{\#\{j\colon w_j=u_j\}}f(\vecw)\geq 0,
\]
where the sum is taken over all $\vecw$ such that $w_j=u_j$ or $u_j+h_j$ 
for all $j\in\{1,\ldots,d\}$.

Conversely, suppose that $\partial_{1,\ldots,d}f(\vecu)\geq 0$.  
By the standard calculus, for all $\vecu, \vecv\in [0,1]^d$ with $u_j<v_j$ for all 
$j\in\{1,\ldots,d\}$, we have 
\[
  \sum (-1)^{\#\{j\colon w_j=u_j\}}f(\vecw)
  =\int_{u_1}^{v_1}\cdots\int_{u_d}^{v_d}\partial_{1,\ldots,d}f(t_1,\ldots,t_d)\,{\d}t_1\dots{\d}t_d
    \geq 0,
\]
where the sum is taken over all $\vecw$ such that $w_j=u_j$ or $v_j$ for all $j\in\{1,\ldots,d\}$.
\qed

\section{Performance measures of copula estimators}
\label{app:estim}

Given a copula estimator $\hat{C}_n$, we consider in Section~\ref{sec:simul} 
the following three performance measures:
\begin{align*}
  \text{integrated squared bias:} \quad
    & \int_{[0, 1]^d} \Bigl[ \E\{ \hat{C}_n( \vecu ) - C( \vecu ) \} \Bigr]^2 \, \d \vecu; \\
  \text{integrated variance:} \quad
    & \int_{[0, 1]^d} \E\Bigl[ \bigl\{ \hat{C}_n( \vecu ) - \E( \hat{C}_n( \vecu) ) \bigr\}^2 \Bigr] \, \d \vecu; \\
  \text{integrated mean squared error:} \quad 
    & \int_{[0, 1]^d} \E\Bigl[ \bigl\{ \hat{C}_n( \vecu ) - C( \vecu) \bigr\}^2 \Bigr] \, \d \vecu.
\end{align*}
To compute these, we apply the following trick. Let $\hat{C}_n^{(1)}$ and $\hat{C}_n^{(2)}$ be independent replications of the estimator, and let the random vector $\vecV$ be independent of the two estimators and uniformly distributed on $[0, 1]^d$. Then each of the three performance measures above can be written as a single expectation with respect to $( \hat{C}_n^{(1)}, \hat{C}_n^{(2)}, \vecV )$:
\begin{align*}
  \text{integrated squared bias:} \quad 
    & \E \left[ \prod_{j=1}^2 \bigl\{ \hat{C}_n^{(j)}( \vecV ) - C( \vecV ) \bigr\} \right]; \\
  \text{integrated variance:} \quad 
    & \E \left[ \frac{1}{2} \bigl\{ \hat{C}_n^{(1)}( \vecV ) - \hat{C}_n^{(2)}( \vecV ) \bigr\}^2 \right]; \\
  \text{integrated mean squared error:} \quad 
    & \E \left[ \frac{1}{2} \sum_{j=1}^2 \bigl\{ \hat{C}_n^{(j)}( \vecV ) - C( \vecV ) \bigr\}^2 \right].
\end{align*}
To see these identities, first condition on $\vecV$ and then use the fact that $\hat{C}_n^{(1)}$ and $\hat{C}_n^{(2)}$ are independent of $\vecV$ and are independent random copies of $\hat{C}_n$.

To compute the above three expectations, we rely on Monte Carlo simulation. For a given large integer $L$, we simulate the independent random vectors
\[
  \vecU_{1,\ell}^{(1)}, \ldots, \vecU_{n,\ell}^{(1)}; \vecU_{1,\ell}^{(2)}, \ldots, \vecU_{n,\ell}^{(2)}; \vecV_\ell,
\]
where $\ell \in \{ 1, \ldots, L \}$. The random vectors $\vecU$ are sampled from $C$ and the random vectors $\vecV$ are sampled from the uniform distribution on $[0, 1]^d$. Then we compute the copula estimator under consideration based on the samples $\{\vecU_{i,\ell}^{(1)}\}_{i=1}^n$ and $\{\vecU_{i,\ell}^{(2)}\}_{i=1}^n$ and evaluated at $\vecV_\ell$, yielding $\hat{C}_{n,\ell}^{(1)}( \vecV_\ell )$ and $\hat{C}_{n,\ell}^{(2)}( \vecV_\ell )$, respectively. We also compute the true copula, $C$, at $\vecV_\ell$. Then we compute the desired function of $\hat{C}_{n,\ell}^{(1)}( \vecV_\ell )$, $\hat{C}_{n,\ell}^{(2)}( \vecV_\ell )$ and $C( \vecV_\ell )$ as in the three expectations above. Finally, we average over the $L$ samples to obtain a Monte Carlo estimate of the desired performance measure.

The plots in Figures~\ref{fig:emp-checker-beta} and~\ref{fig:beta-bern3-bernJSV} are based on $L = 20\,000$ replications.

\subsubsection*{Localized integrated mean squared error}
\label{app:lre}
For a copula estimator $\widehat{C}_n$, and a Borel set $B \subset [0, 1]^d$ with positive 
Lebesgue measure $\lvert B \rvert$, we define the \textit{localized integrated mean squared error}
of $\widehat{C}_n$ on $B$ by
\[
  \operatorname{LIMSE}( \widehat{C}_n; B )
  =
  \frac{1}{\lvert B \rvert} \int_B 
      \E \left[\bigl\lbrace \widehat{C}_n( \vecv ) - C( \vecv ) \bigr\rbrace^2  \right]
  \, \d \vecv 
  =
  \E \left[ \bigl\lbrace \widehat{C}_n( \vecV ) - C( \vecV ) \bigr\rbrace^2 \right],
\]
with $\vecV$ uniformly distributed on $B$ and independent of the random sample $\vecU_1, \ldots, \vecU_n$ underlying the copula estimator $\widehat{C}_n$.  It can be computed by Monte Carlo simulation and integration via
\begin{equation}
\label{eq:LIMSE-MC}
  \widehat{\operatorname{LIMSE}}_L( \widehat{C}_n; B )
  =
  \frac{1}{L} \sum_{\ell=1}^L \bigl\lbrace \widehat{C}_{n,\ell}( \vecV_\ell ) - C( \vecV_\ell ) \bigr\rbrace^2
\end{equation}
where $\hat{C}_{n,\ell}$ denotes the copula estimator based upon a random $n$-sample $\vecU_1^{(\ell)}, \ldots, \vecU_n^{(\ell)}$ from $C$ and where $\vecV_\ell$ is uniformly distributed on $B$, and all random vectors are independent.  Here we are concerned only with the integrated mean squared error, so we dispense with the computational trick above.  The plots in Figure~\ref{fig:LRE} are based on $L = 20\,000$ replications.

\clearpage

\begin{figure}
\begin{center}
\begin{tabular}{@{}c@{}c@{}c@{}}
\includegraphics[width=0.33\textwidth]{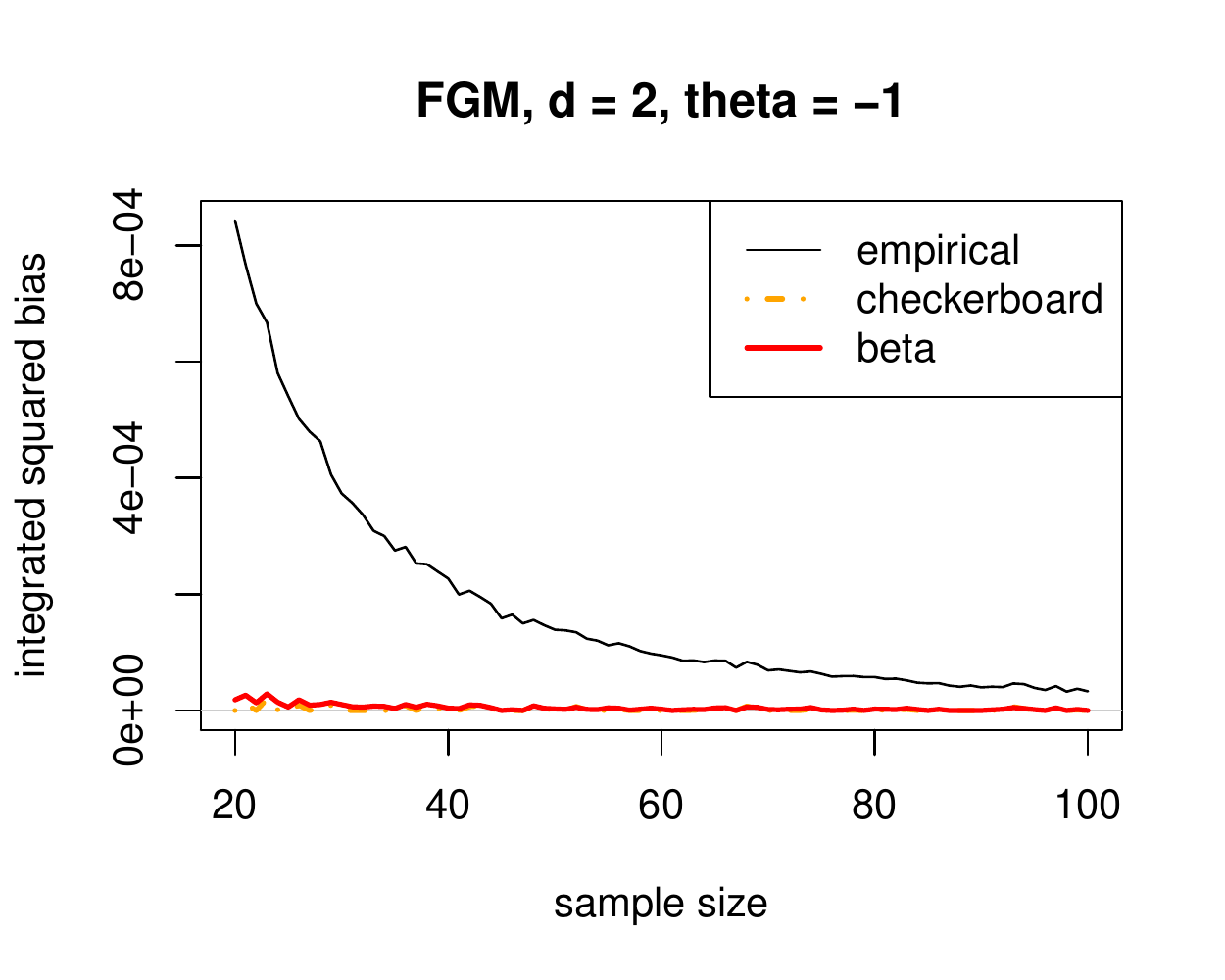}&
\includegraphics[width=0.33\textwidth]{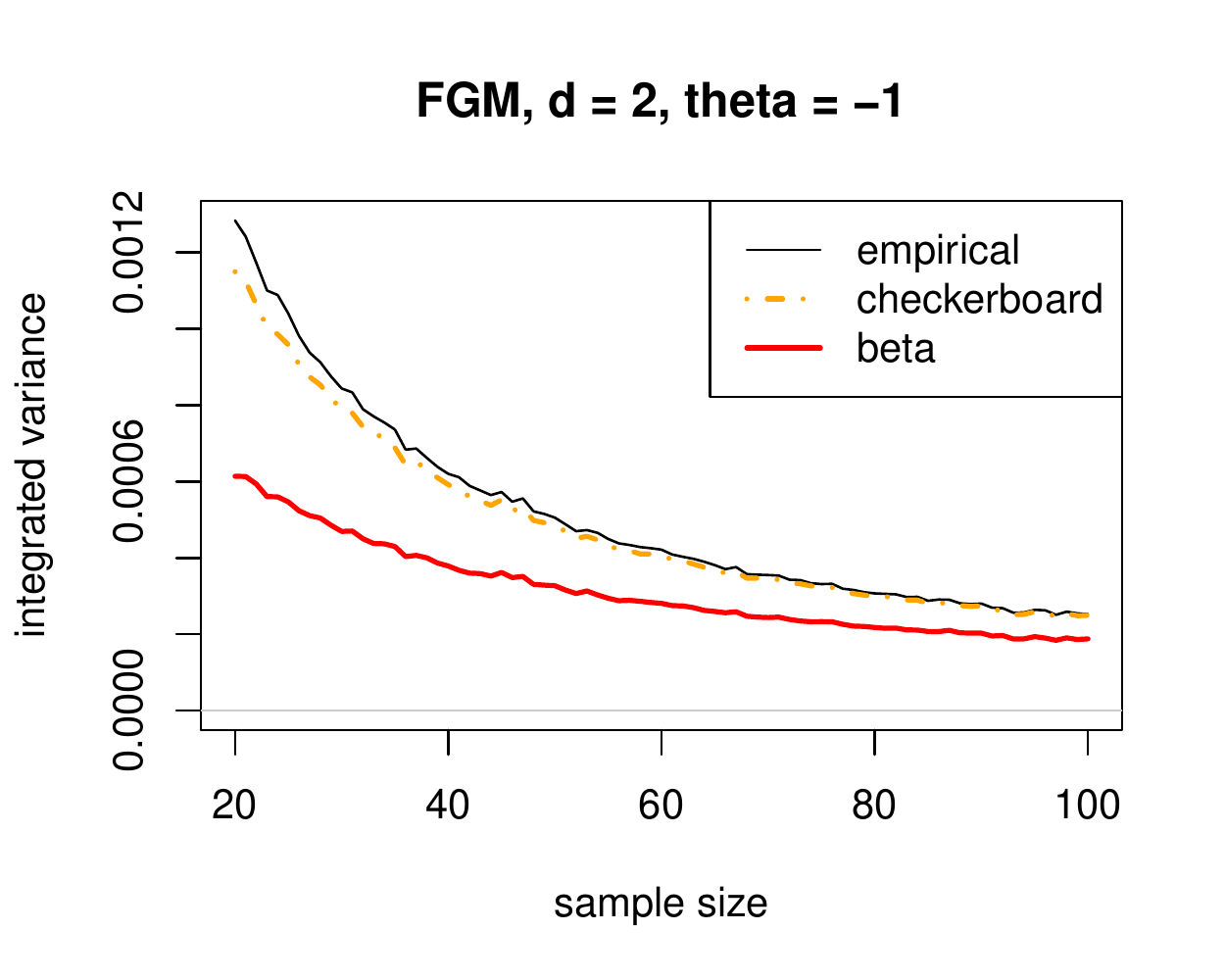}&
\includegraphics[width=0.33\textwidth]{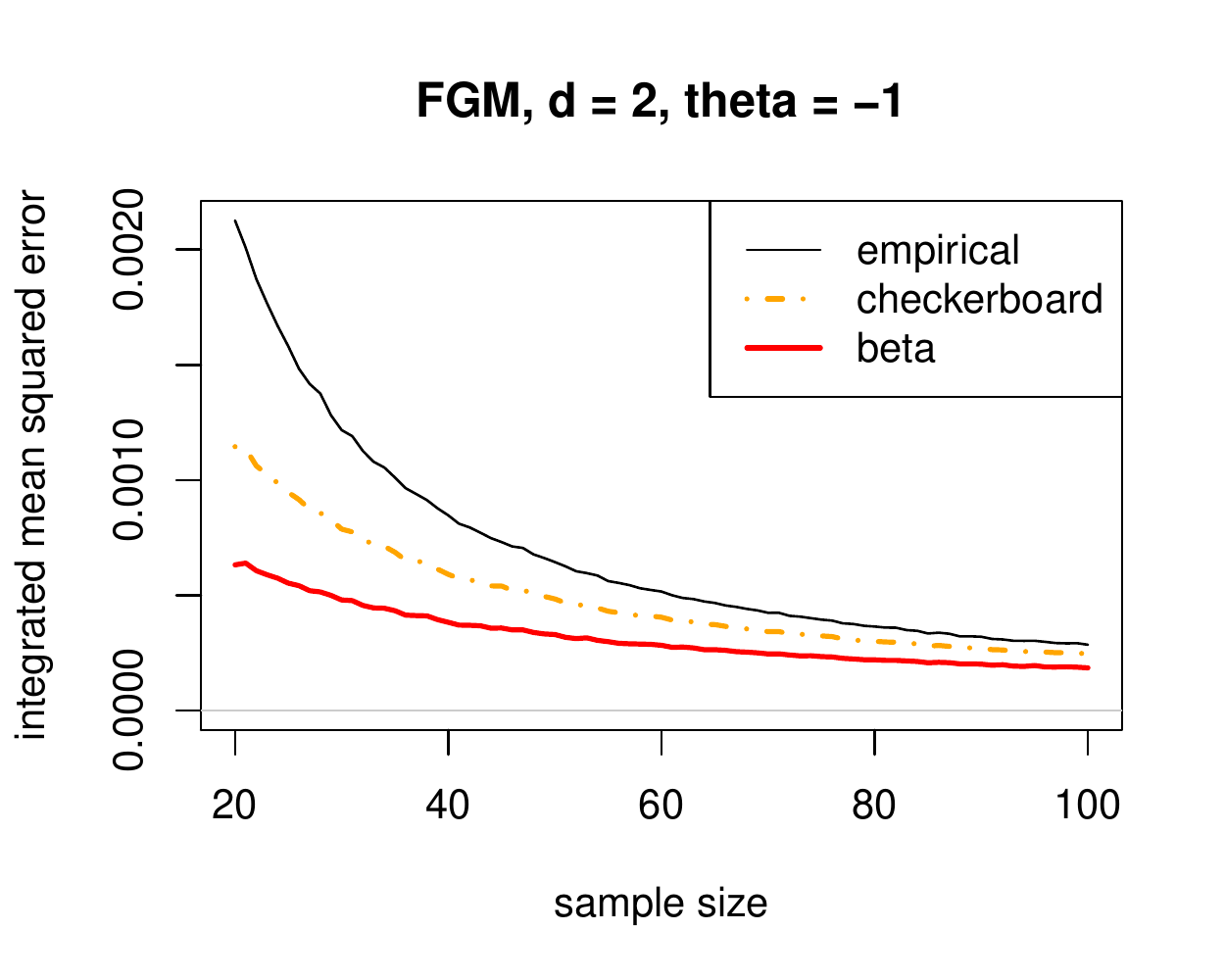}\\
\includegraphics[width=0.33\textwidth]{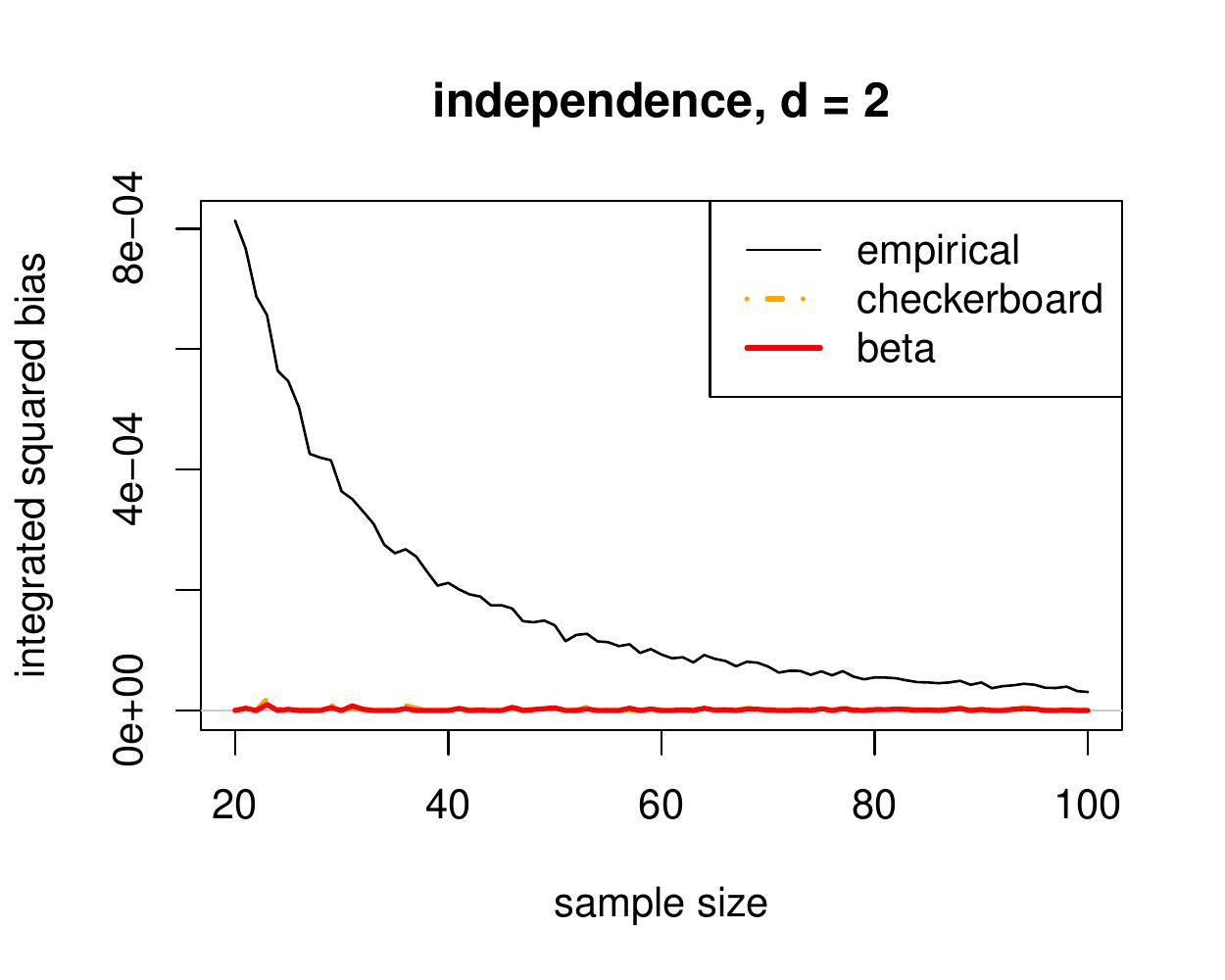}&
\includegraphics[width=0.33\textwidth]{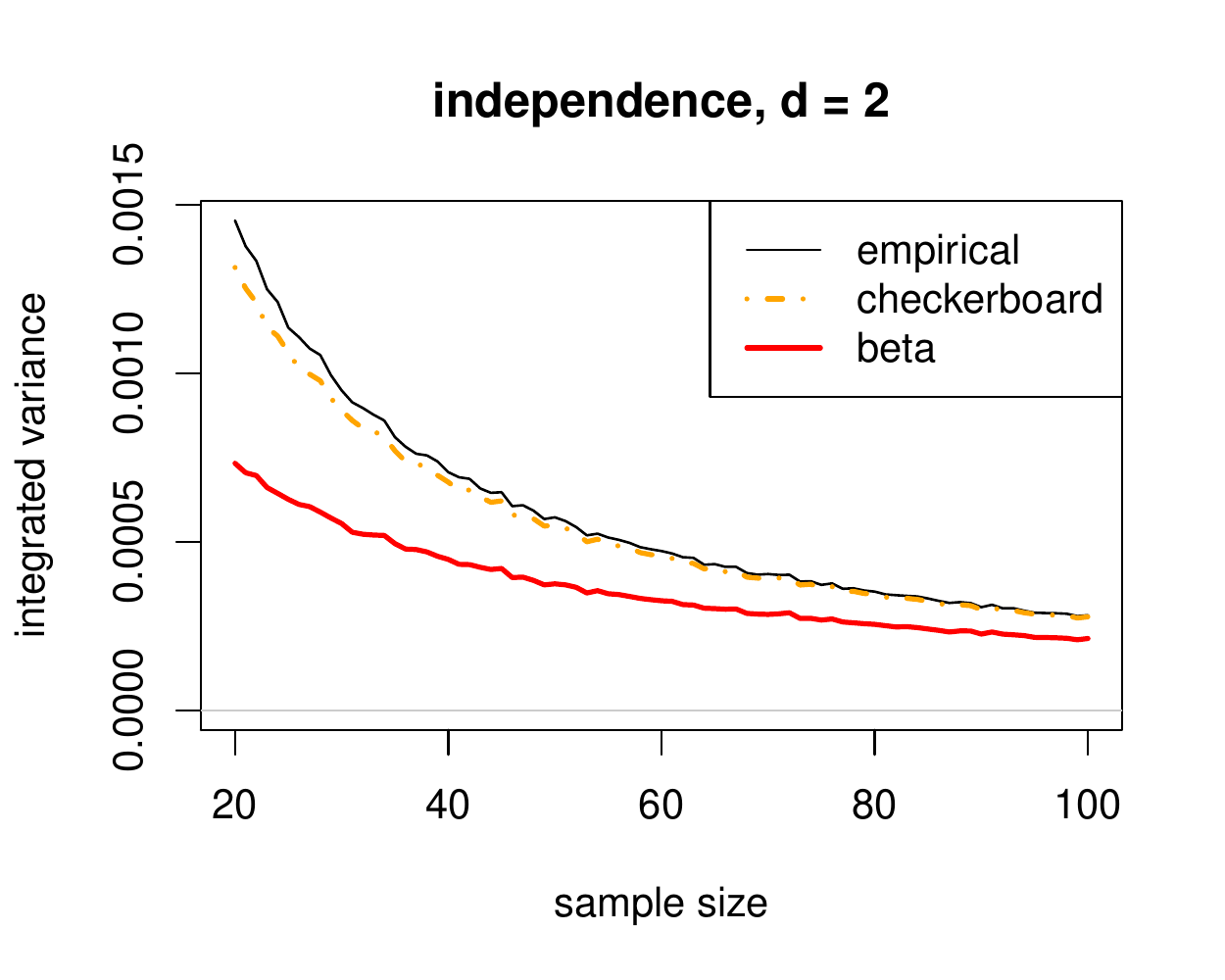}&
\includegraphics[width=0.33\textwidth]{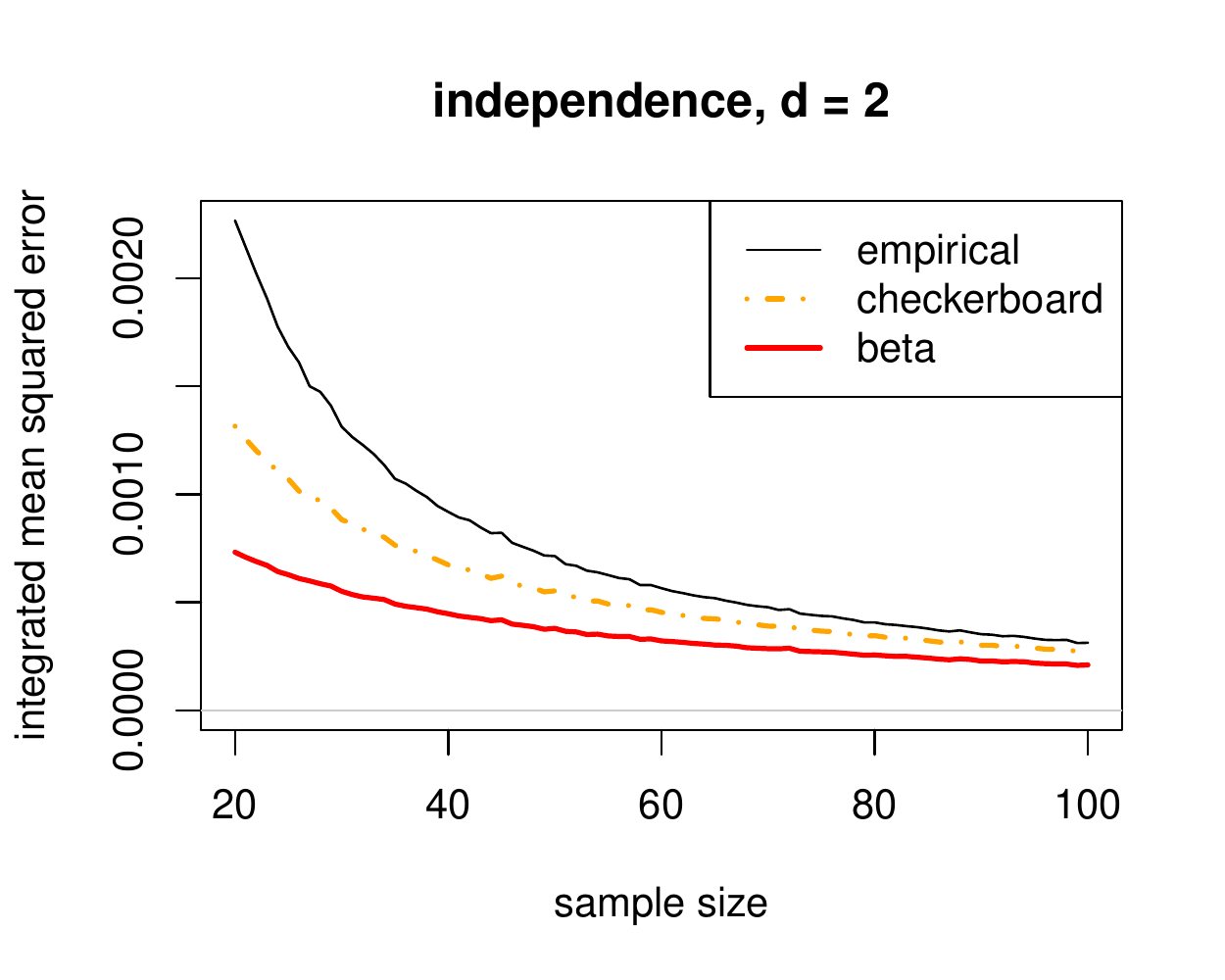}\\
\includegraphics[width=0.33\textwidth]{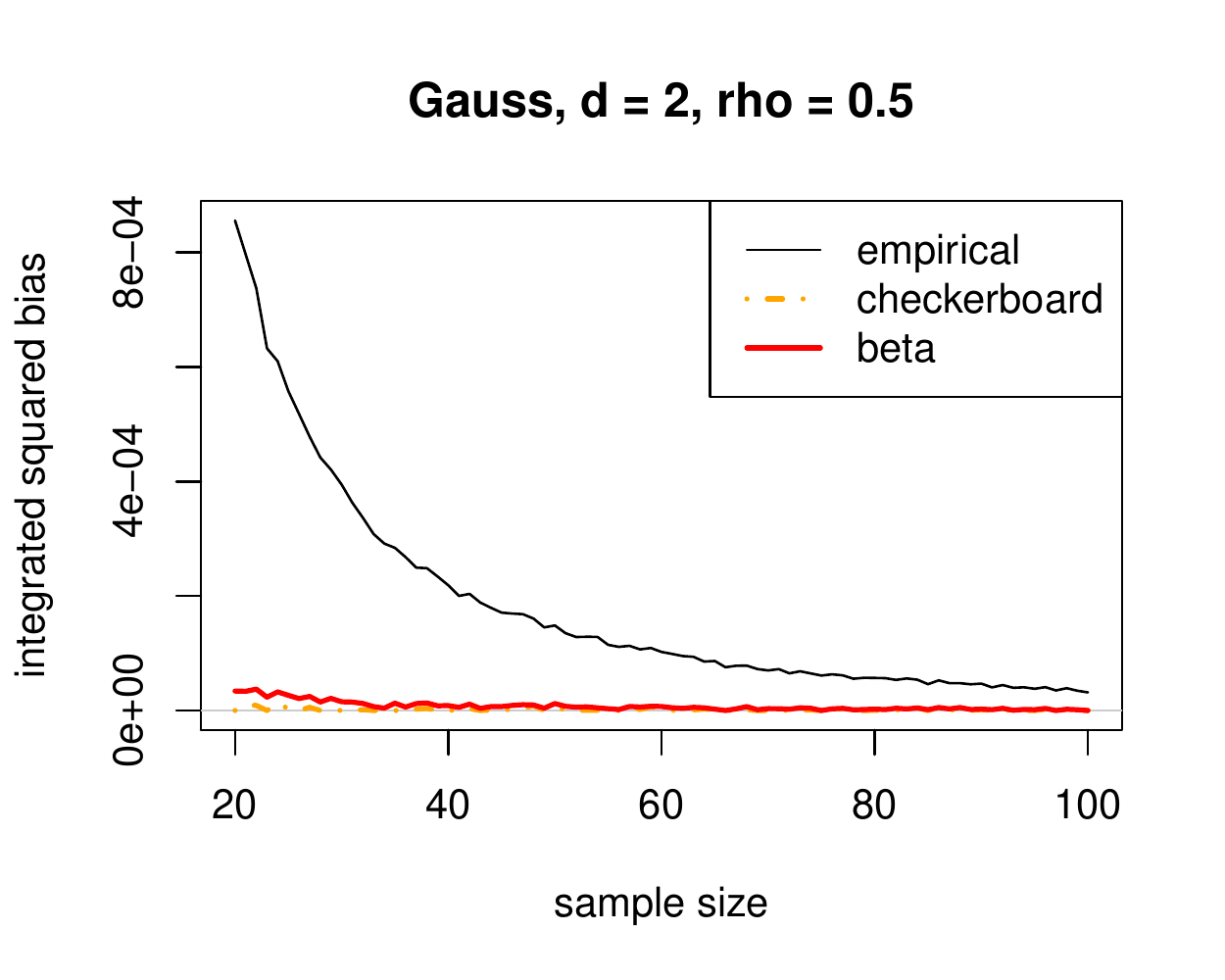}&
\includegraphics[width=0.33\textwidth]{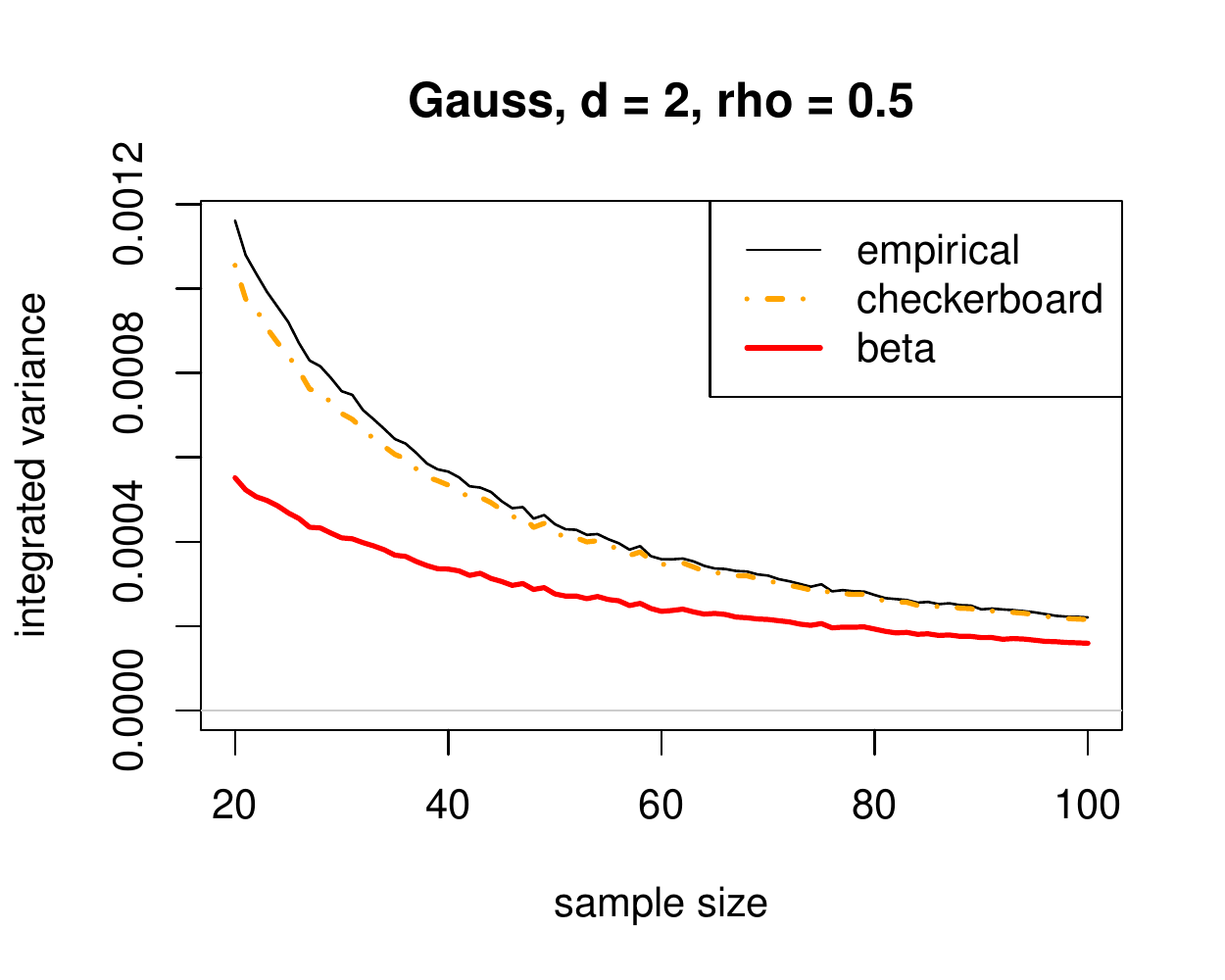}&
\includegraphics[width=0.33\textwidth]{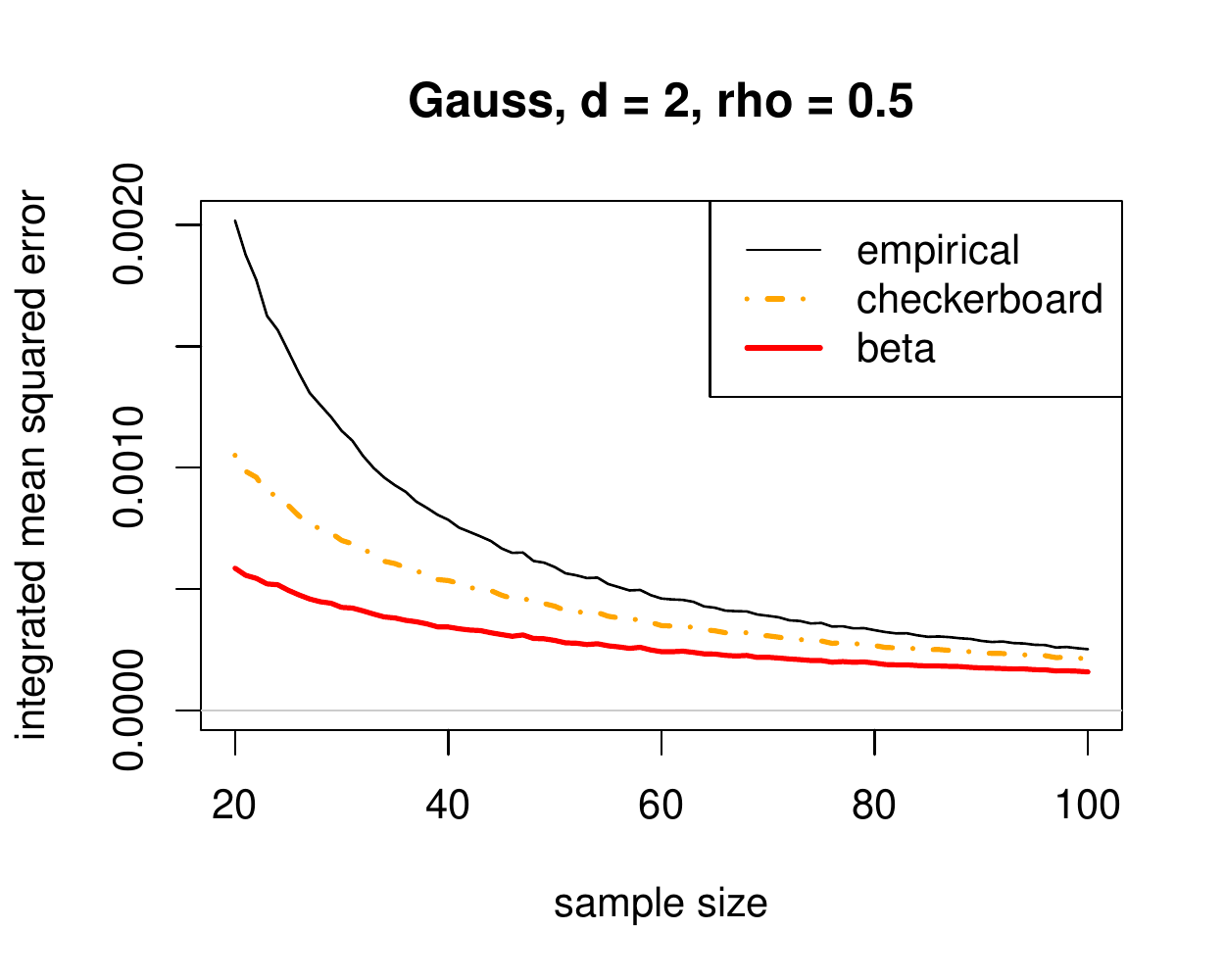}\\
\includegraphics[width=0.33\textwidth]{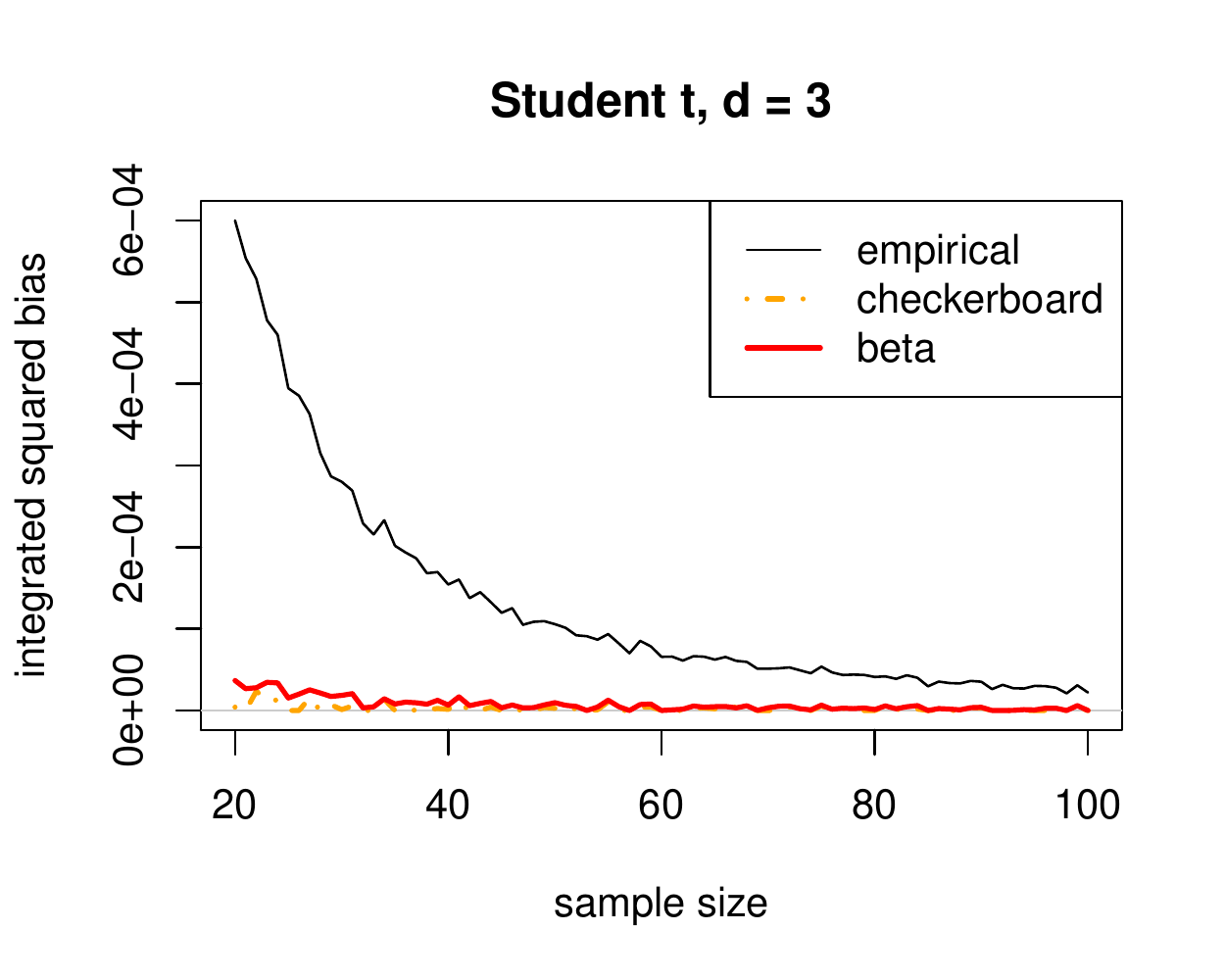}&
\includegraphics[width=0.33\textwidth]{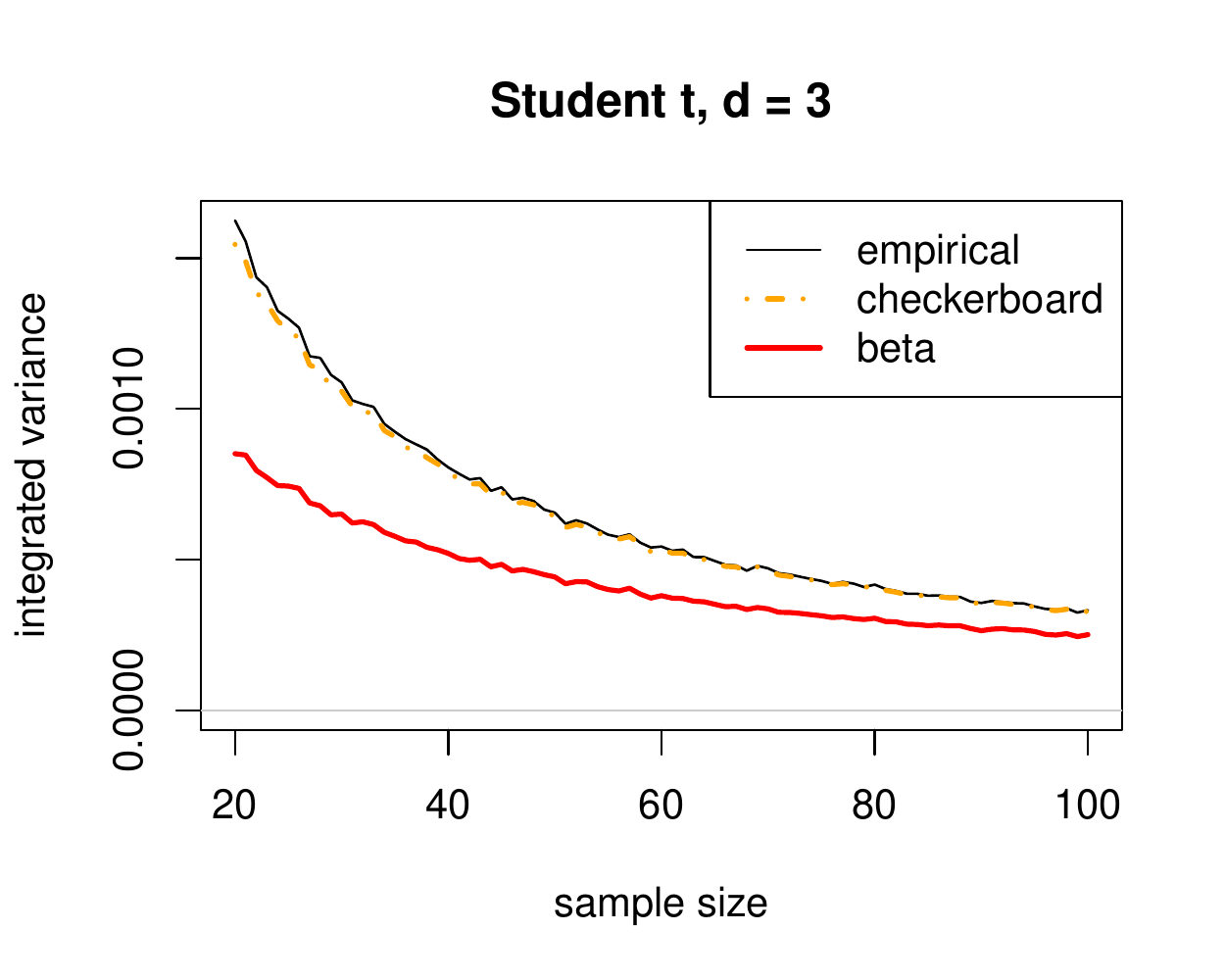}&
\includegraphics[width=0.33\textwidth]{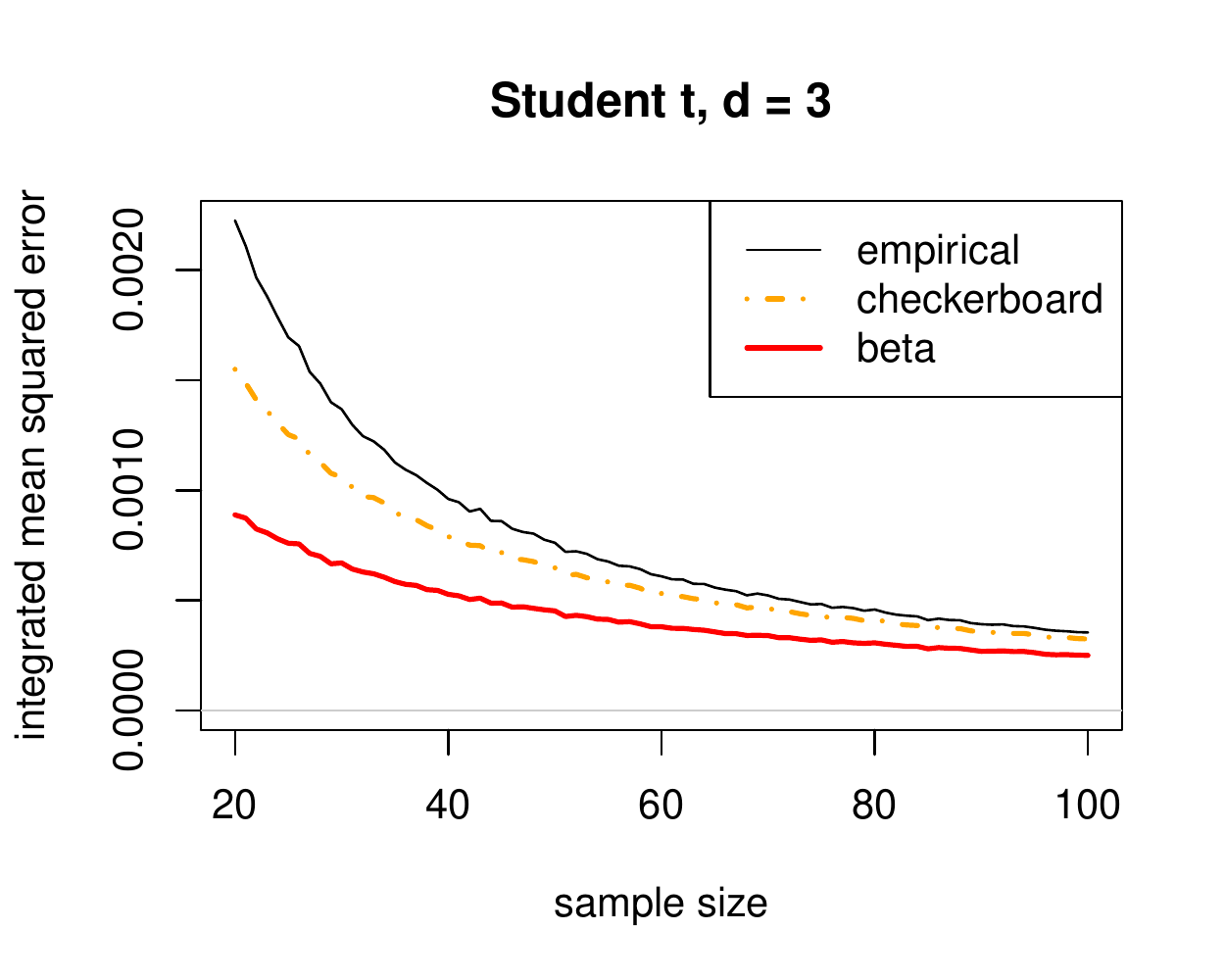}\\
\includegraphics[width=0.33\textwidth]{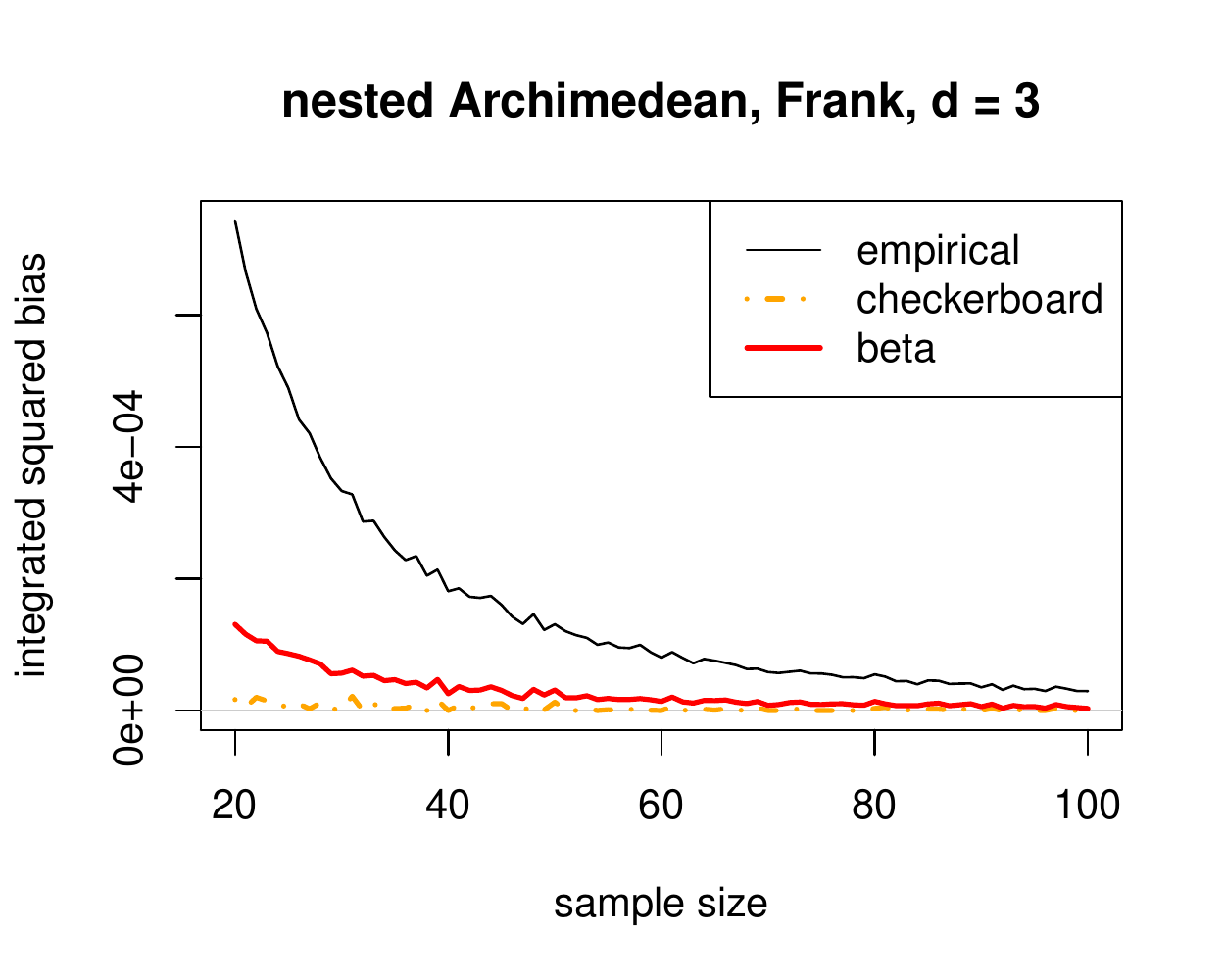}&
\includegraphics[width=0.33\textwidth]{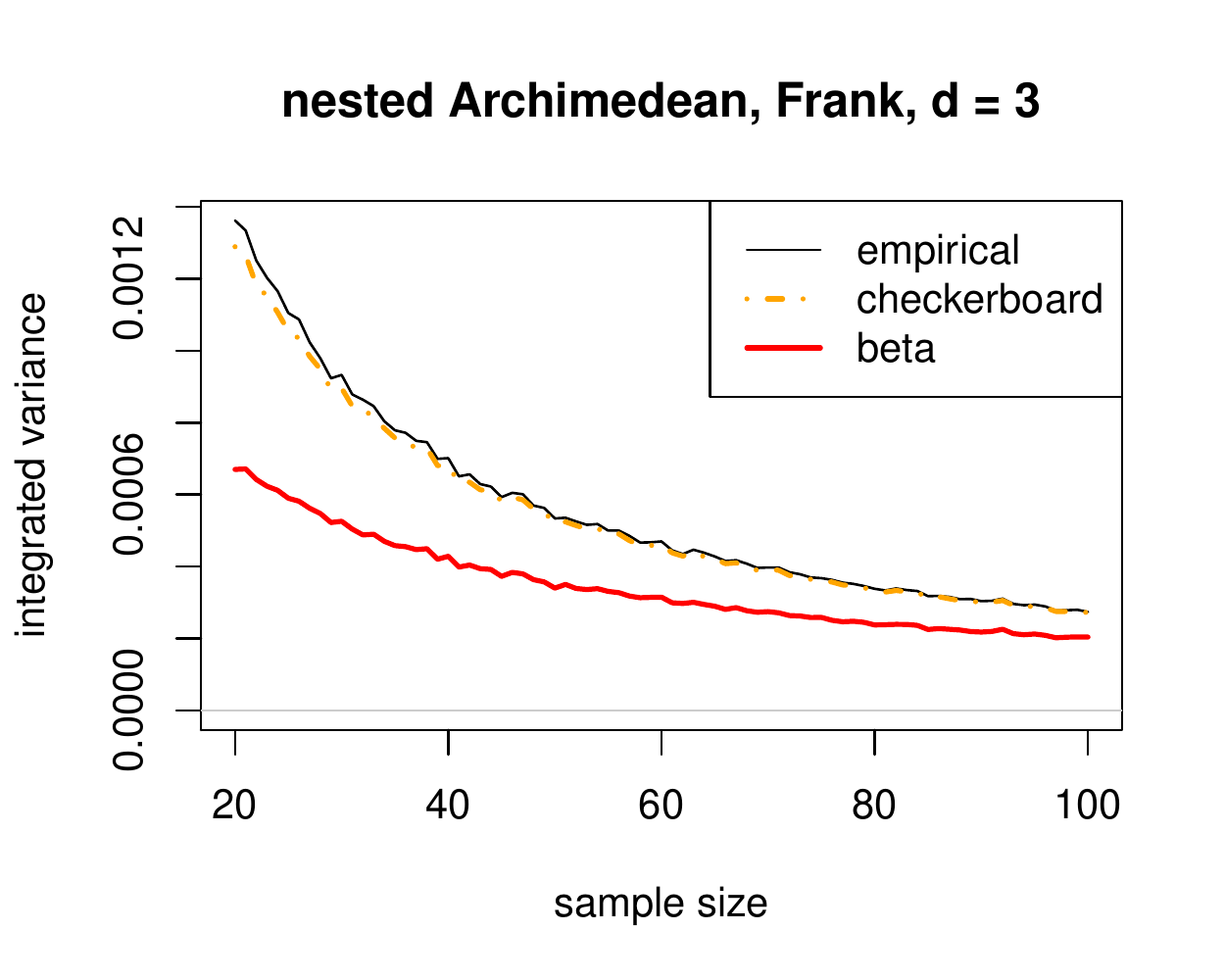}&
\includegraphics[width=0.33\textwidth]{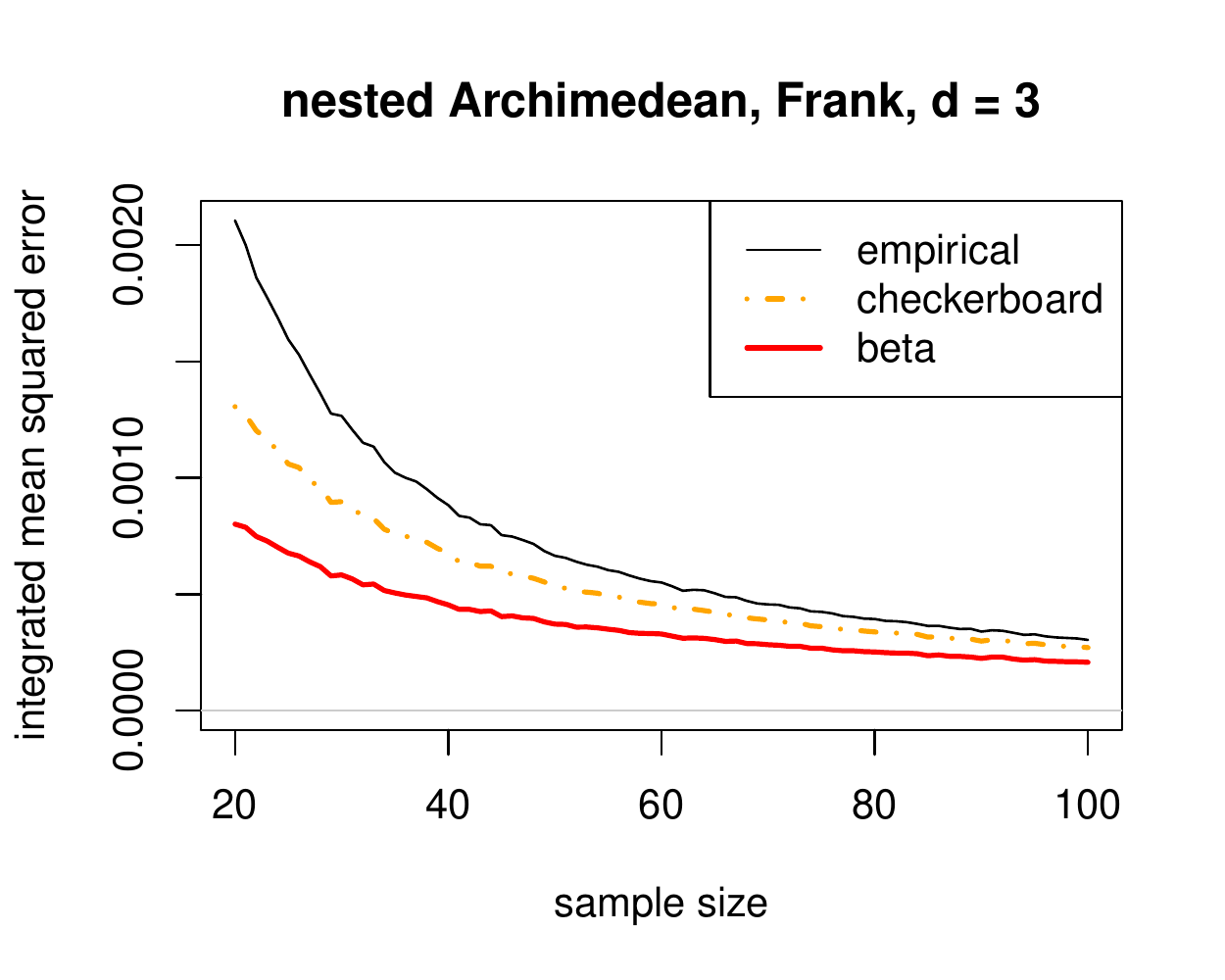}
\end{tabular}
\end{center}
\caption{\label{fig:emp-checker-beta}
Empirical vs empirical checkerboard vs empirical beta copulas. From the top row: FGM ($\theta=-1$), independence, Gauss ($\rho=0.5$), trivariate Student t, and trivariate nested Archimedean.  Left: integrated squared bias. Middle: integrated variance. Right: integrated mean squared error. See Section~\ref{sec:simul} and Appendix~\ref{app:estim} for details. Based on 20\,000 Monte Carlo replications.}
\end{figure}

\clearpage
\begin{figure}
  \begin{center}
    \begin{tabular}{cc}
      \includegraphics[width=.4\textwidth]{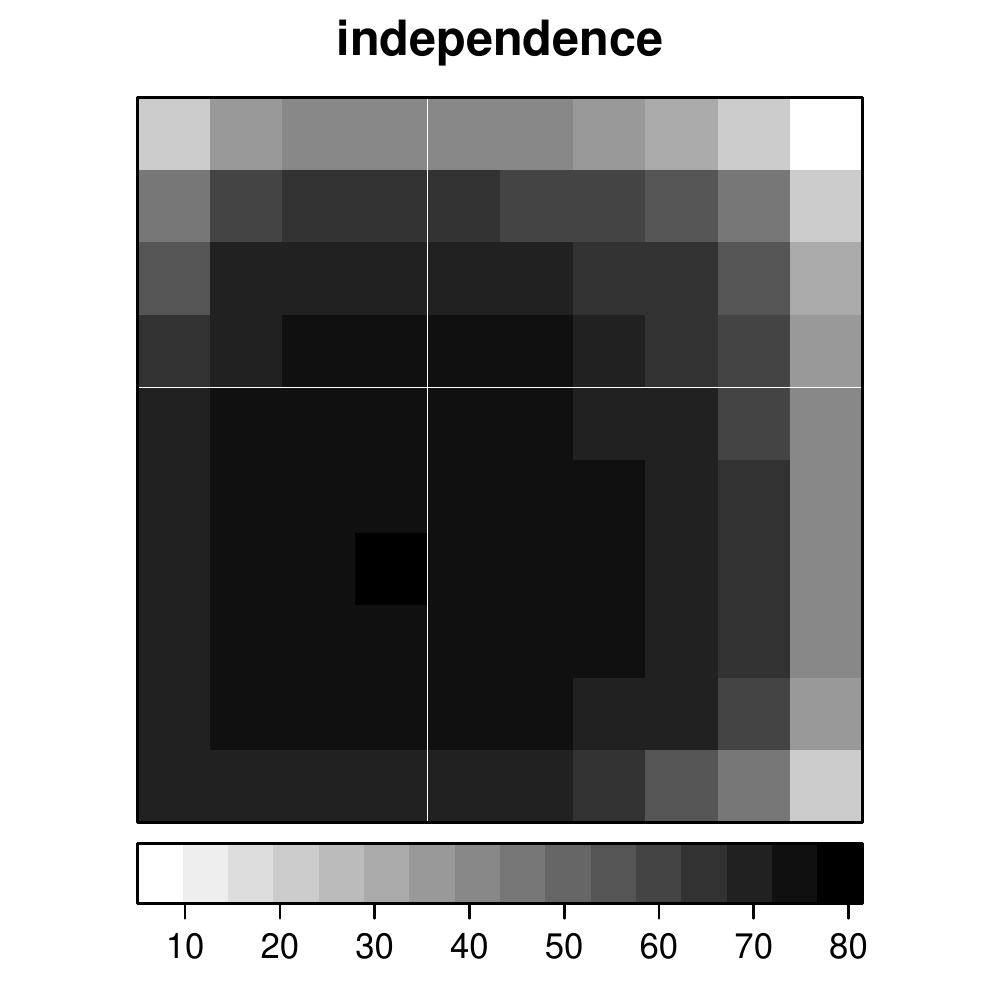} &
      \includegraphics[width=.4\textwidth]{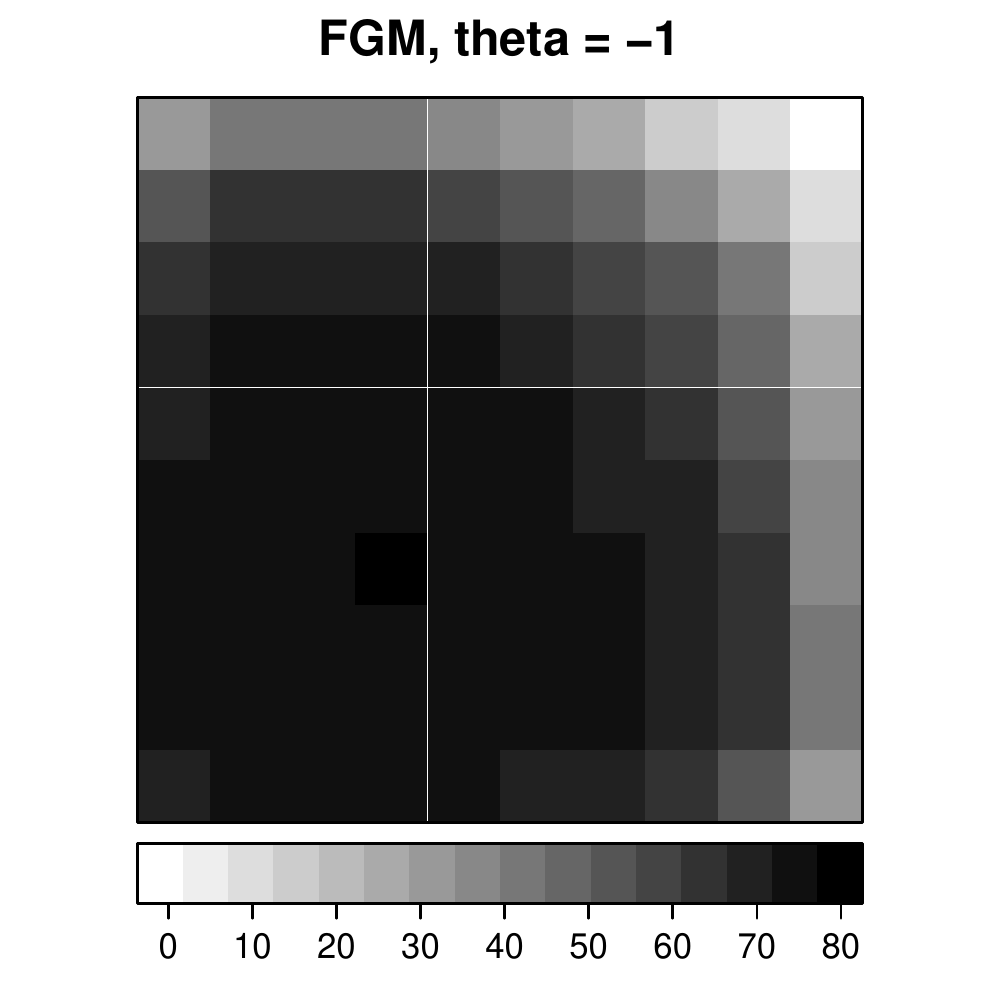} \\
      \includegraphics[width=.4\textwidth]{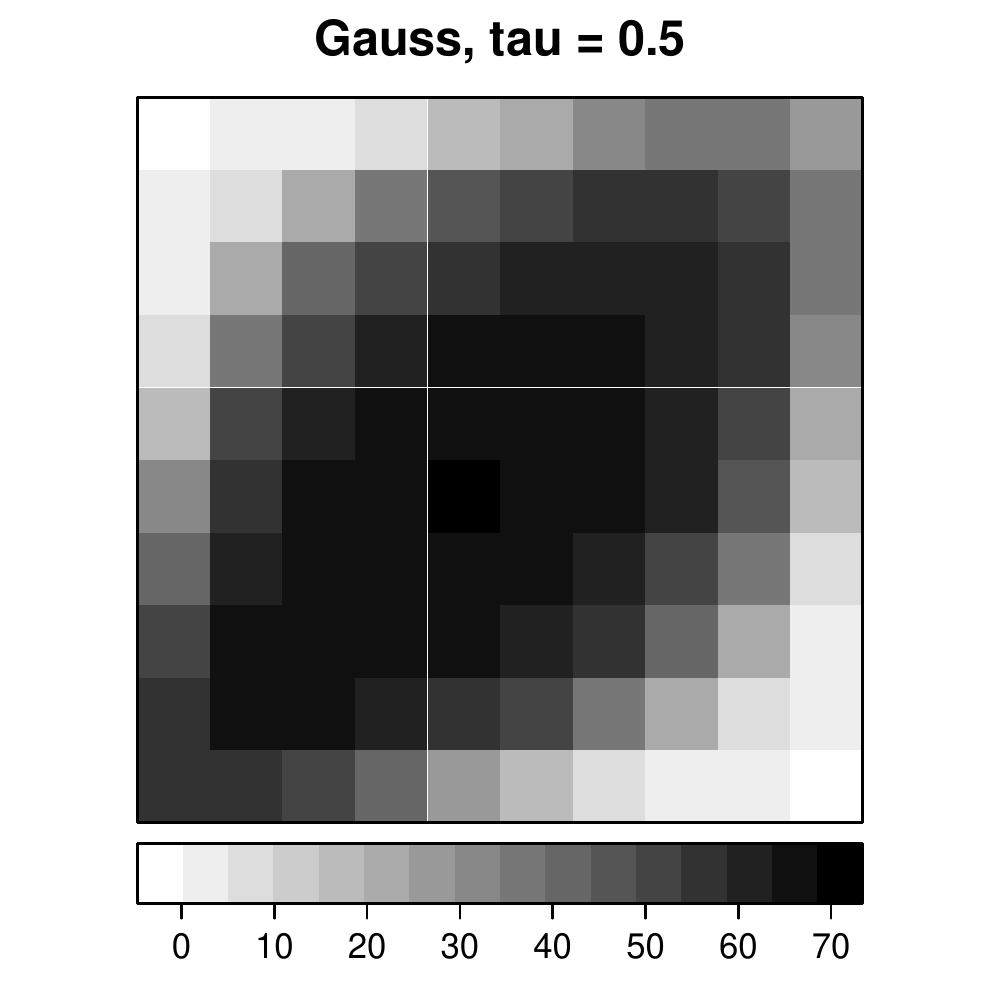} & 
      \includegraphics[width=.4\textwidth]{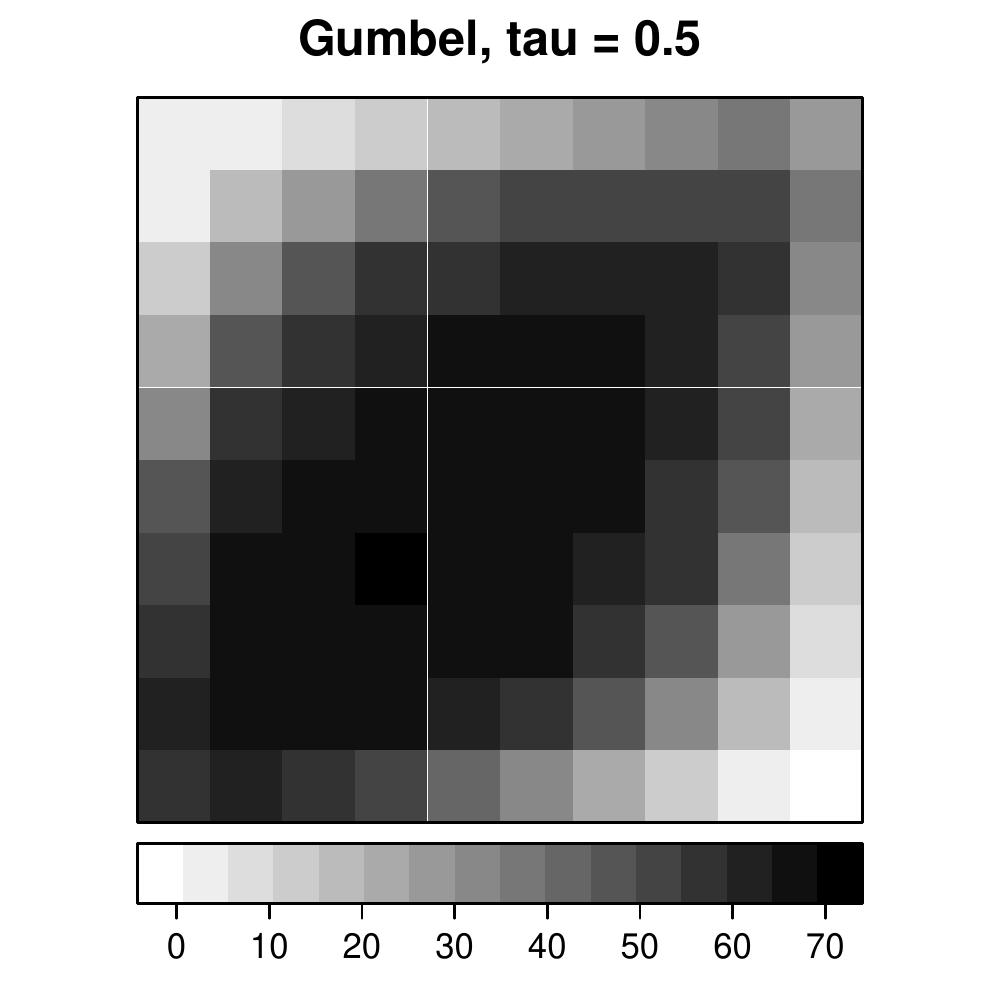}
    \end{tabular}
  \end{center}
  \caption{\label{fig:LRE} Localized relative efficiencies \eqref{eq:LRE} of the empirical copula with respect to the empirical beta copula at sample size $n = 100$ based on a partition of the unit square in $10 \times 10$ square cells $B$ and for random samples drawn from four bivariate copulas. Computations based upon 20\,000 Monte Carlo replications.  See Section~\ref{sec:simul} and Appendix~\ref{app:estim} for details.}
\end{figure}

\begin{figure}
\begin{tabular}{@{}c@{}c@{}c@{}}
\includegraphics[width=0.33\textwidth]{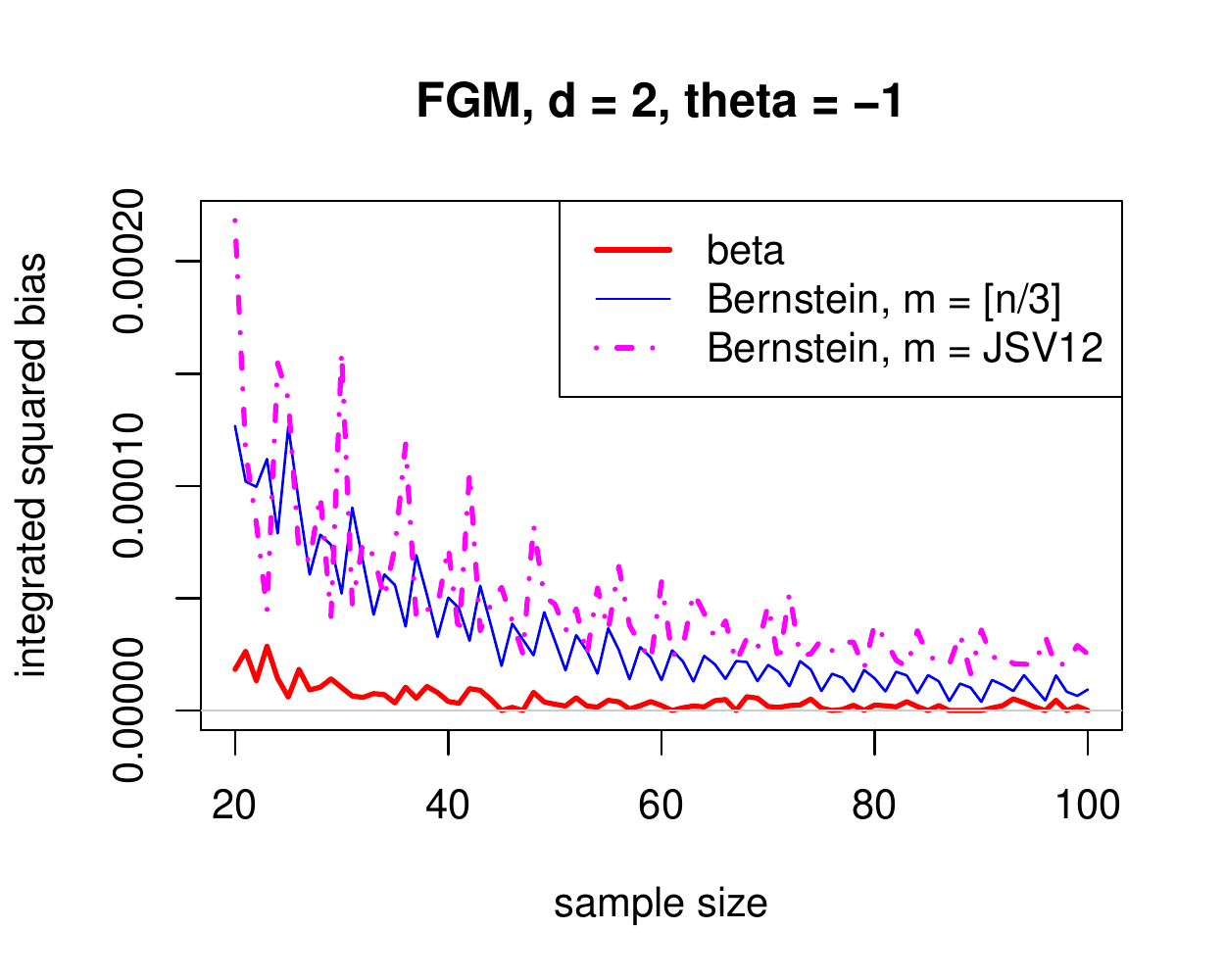}&
\includegraphics[width=0.33\textwidth]{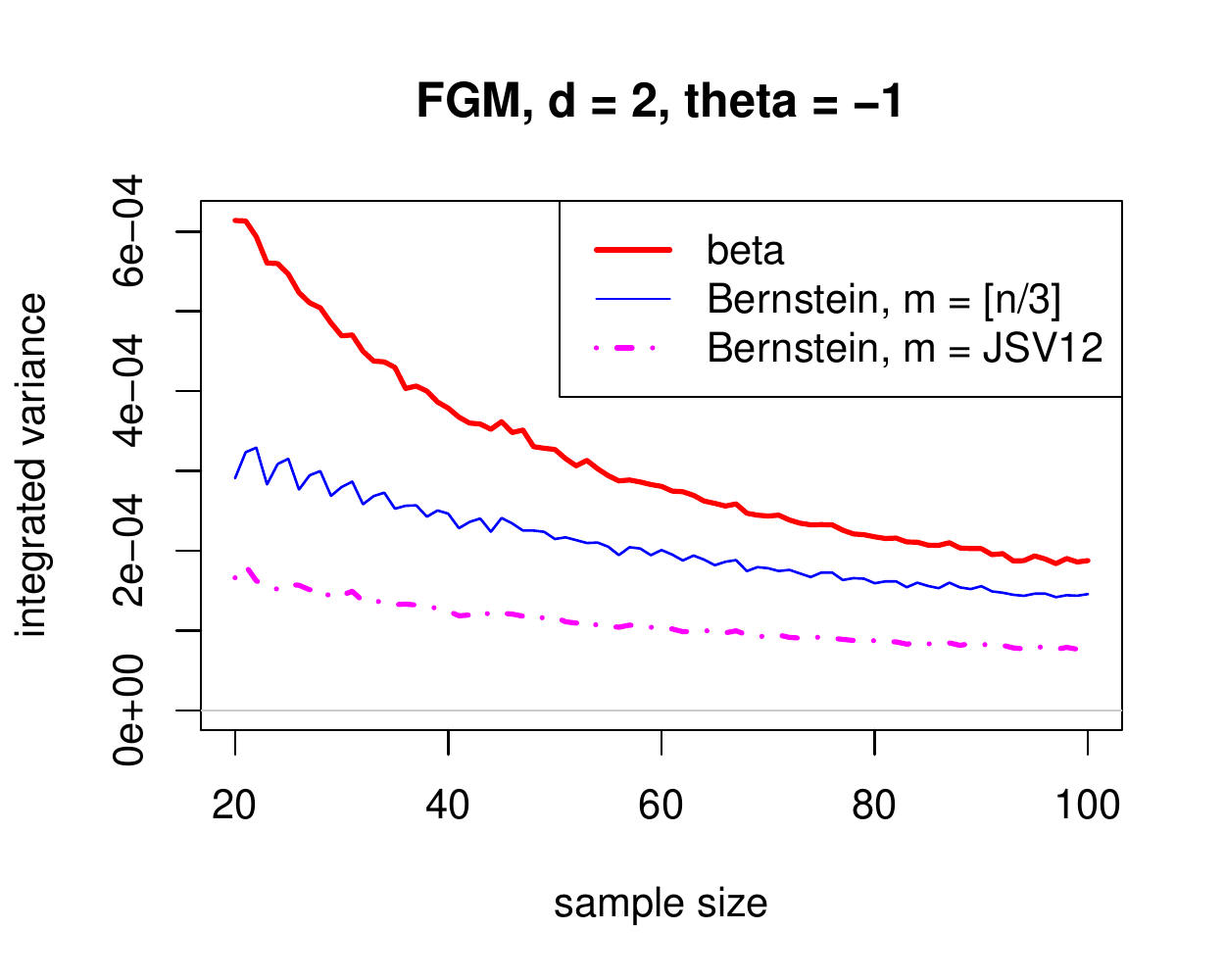}&
\includegraphics[width=0.33\textwidth]{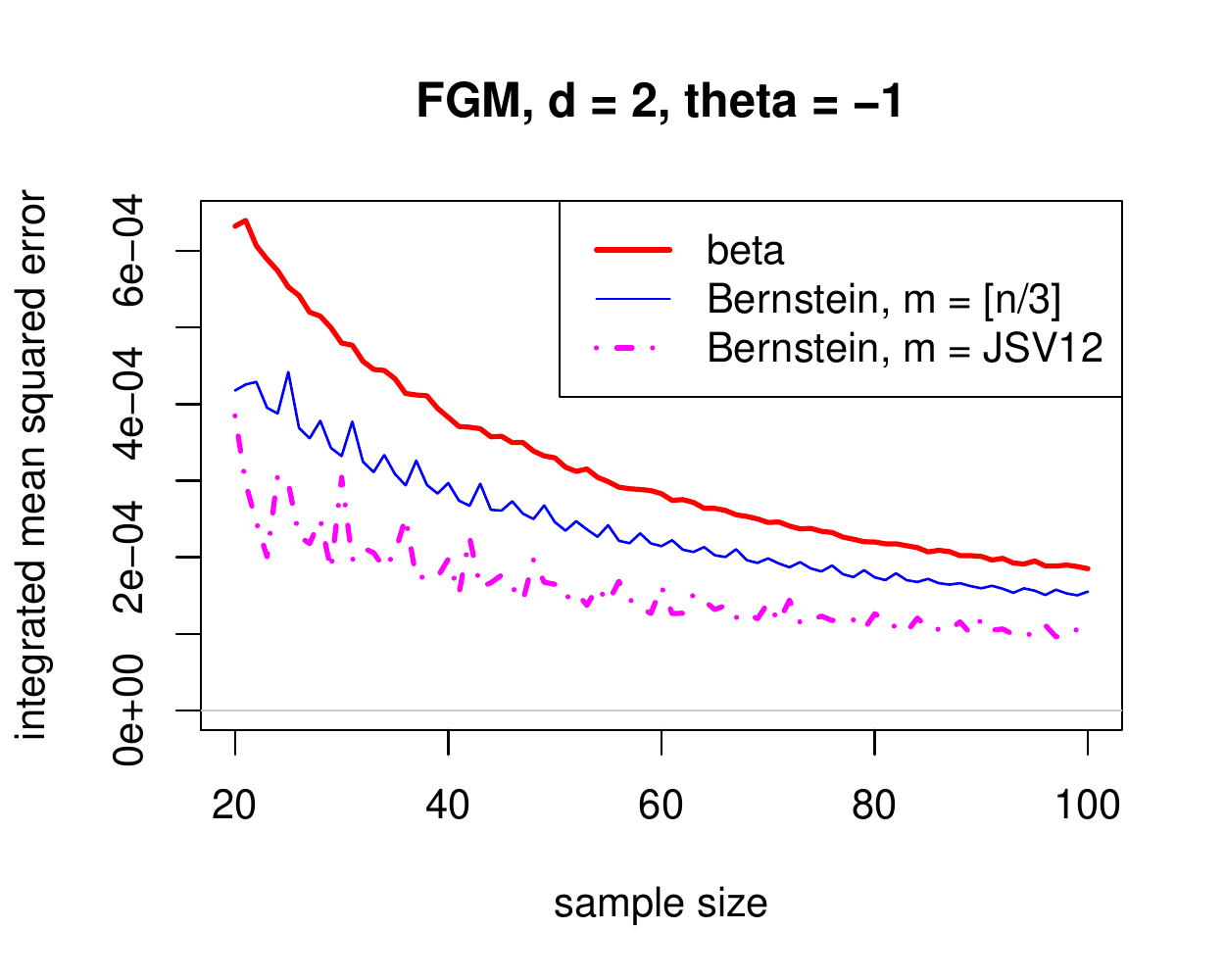}\\
\includegraphics[width=0.33\textwidth]{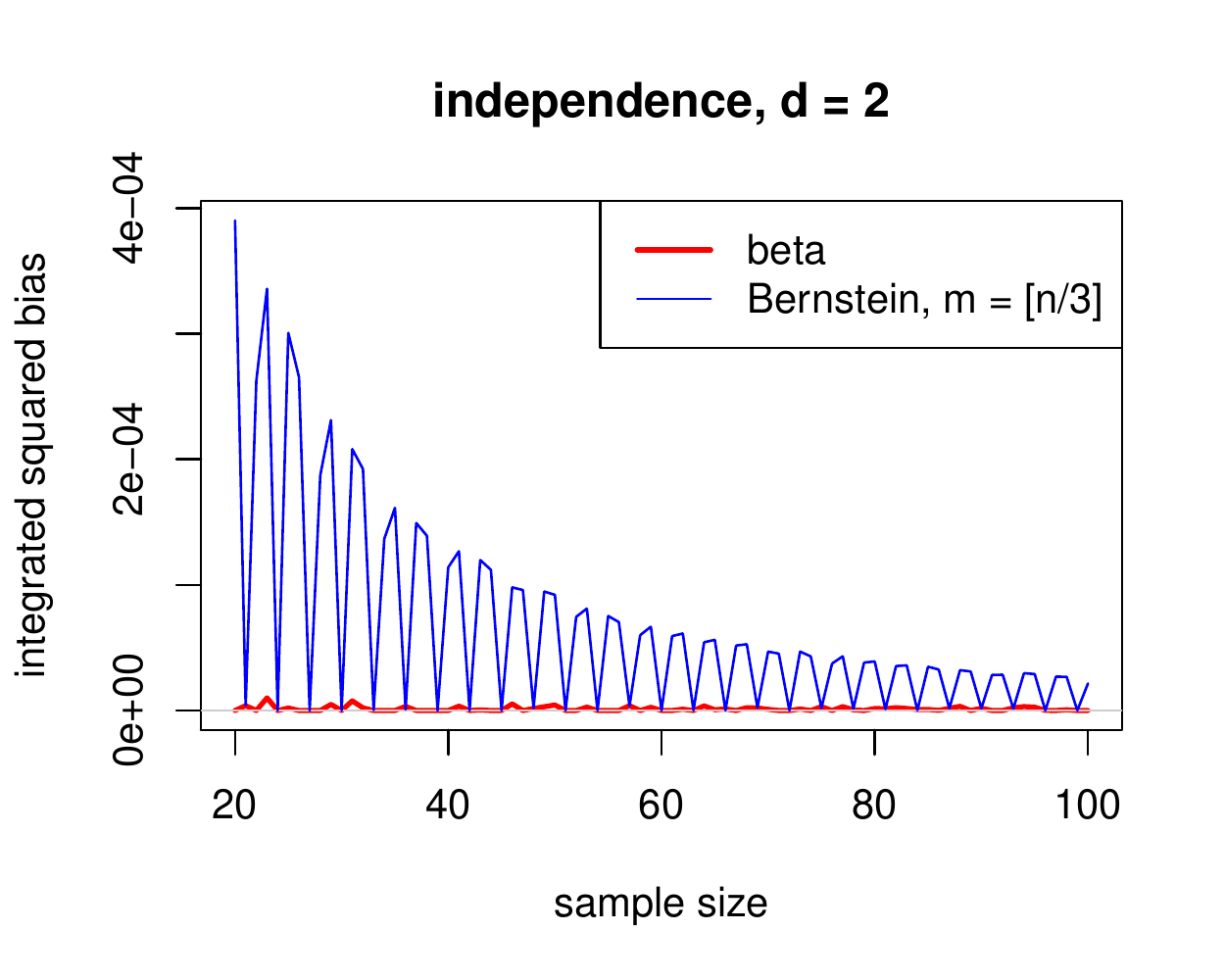}&
\includegraphics[width=0.33\textwidth]{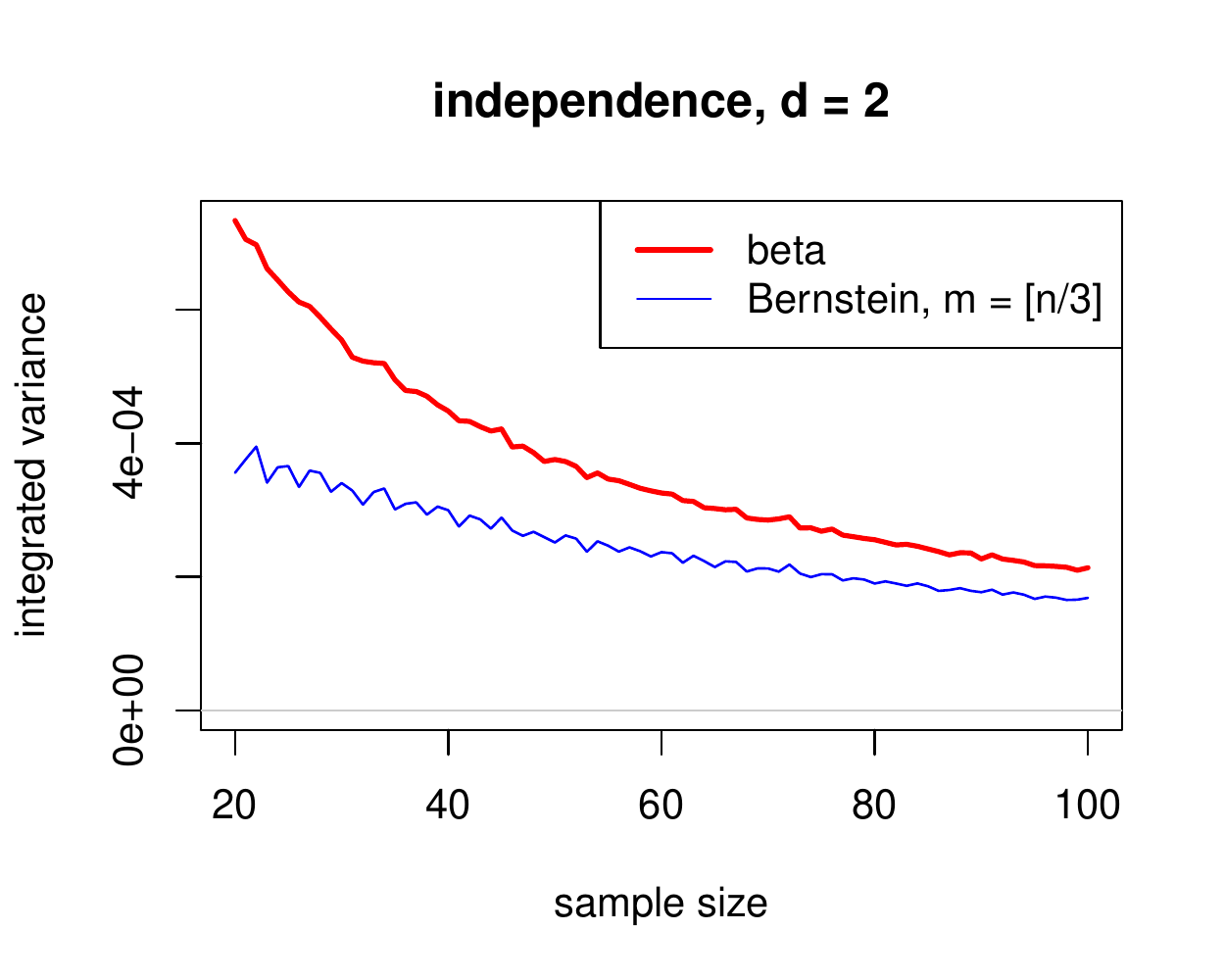}&
\includegraphics[width=0.33\textwidth]{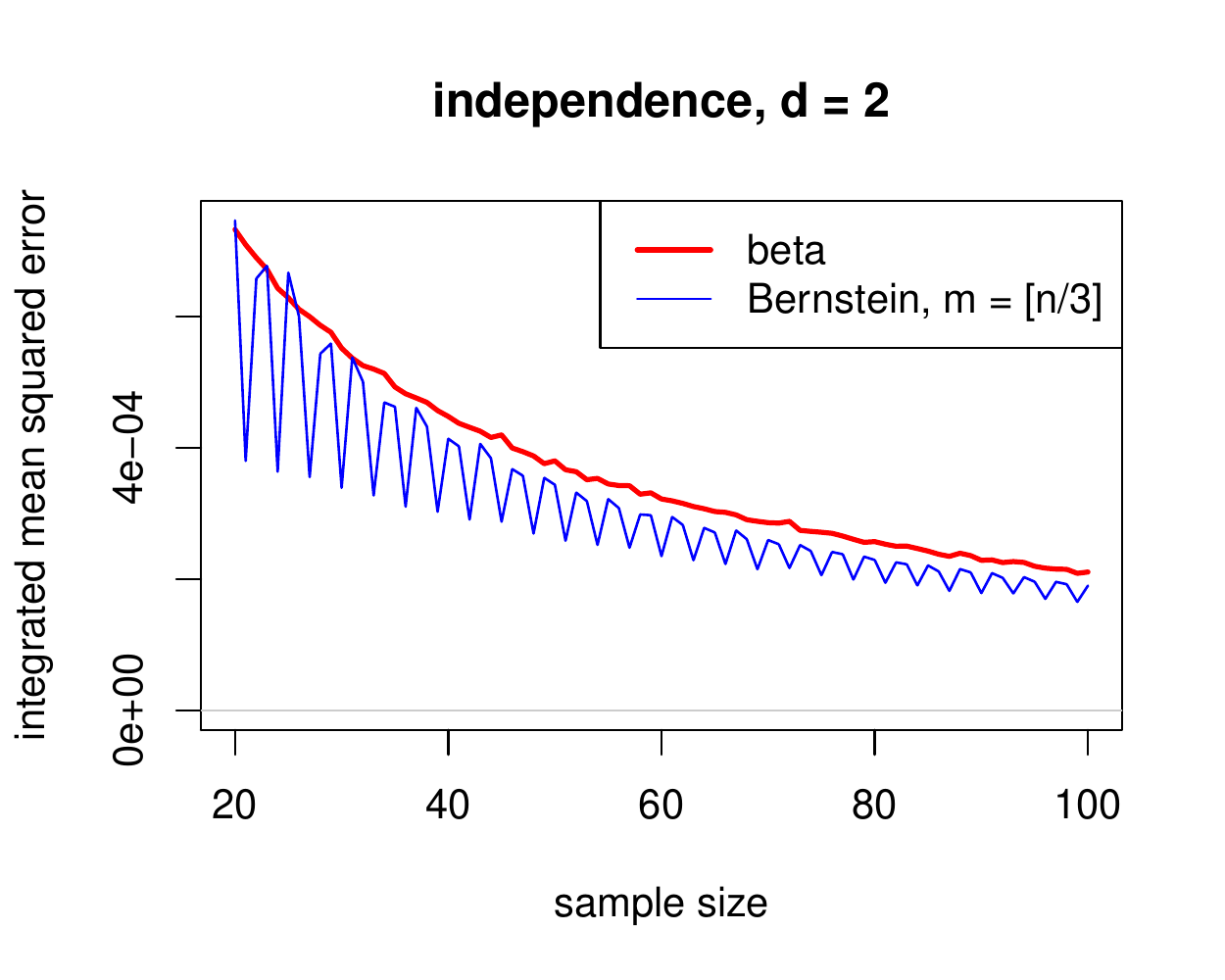}\\
\includegraphics[width=0.33\textwidth]{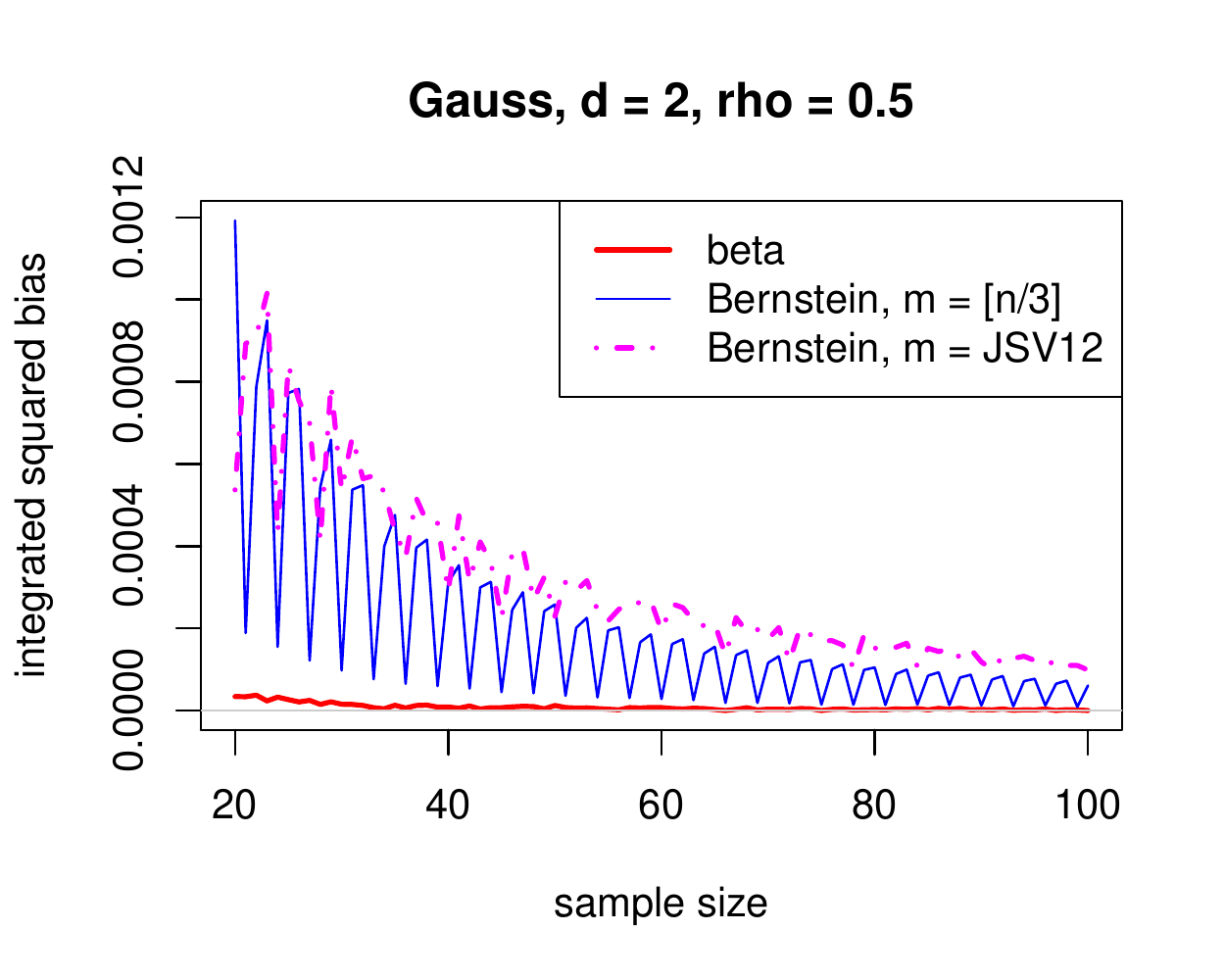}&
\includegraphics[width=0.33\textwidth]{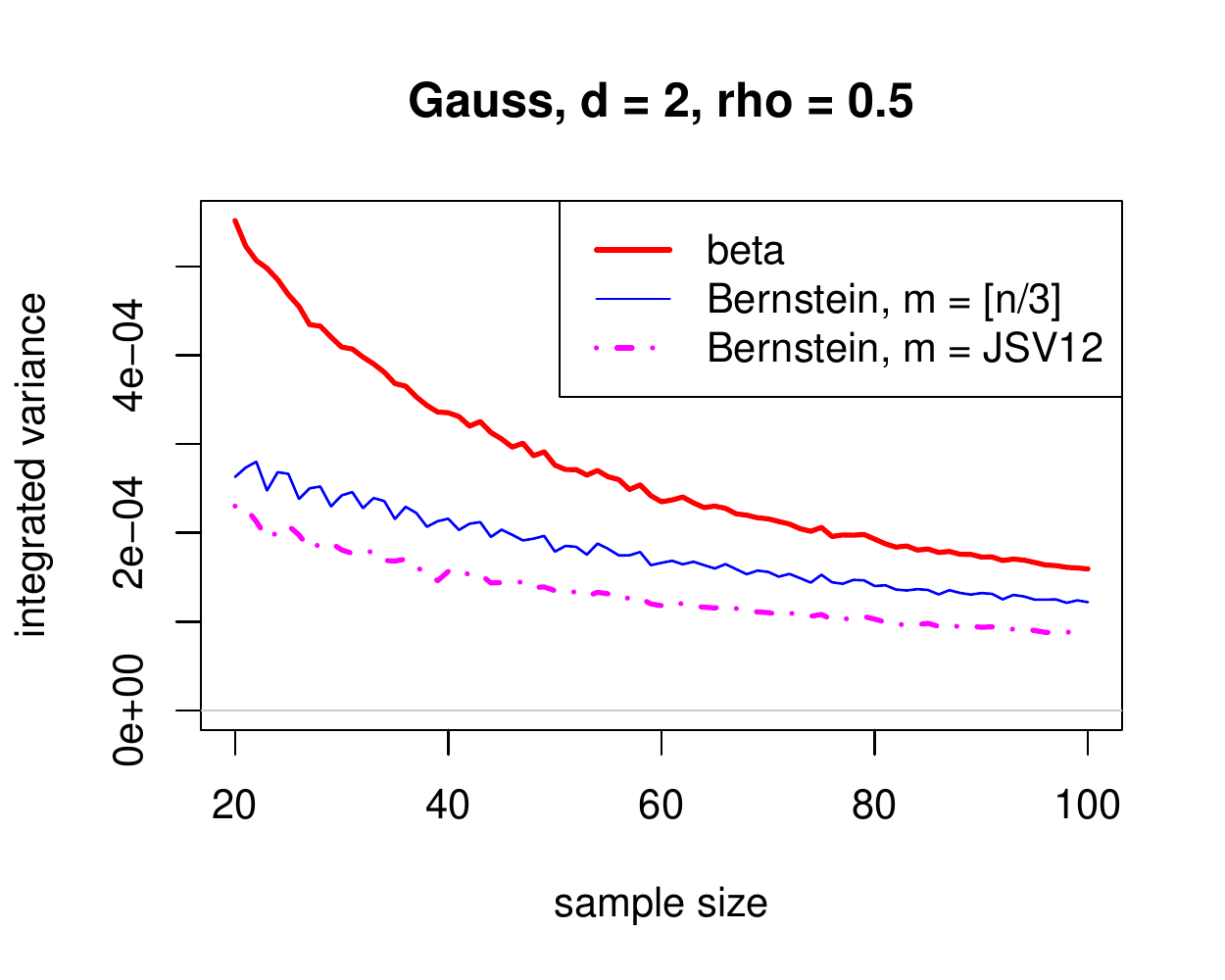}&
\includegraphics[width=0.33\textwidth]{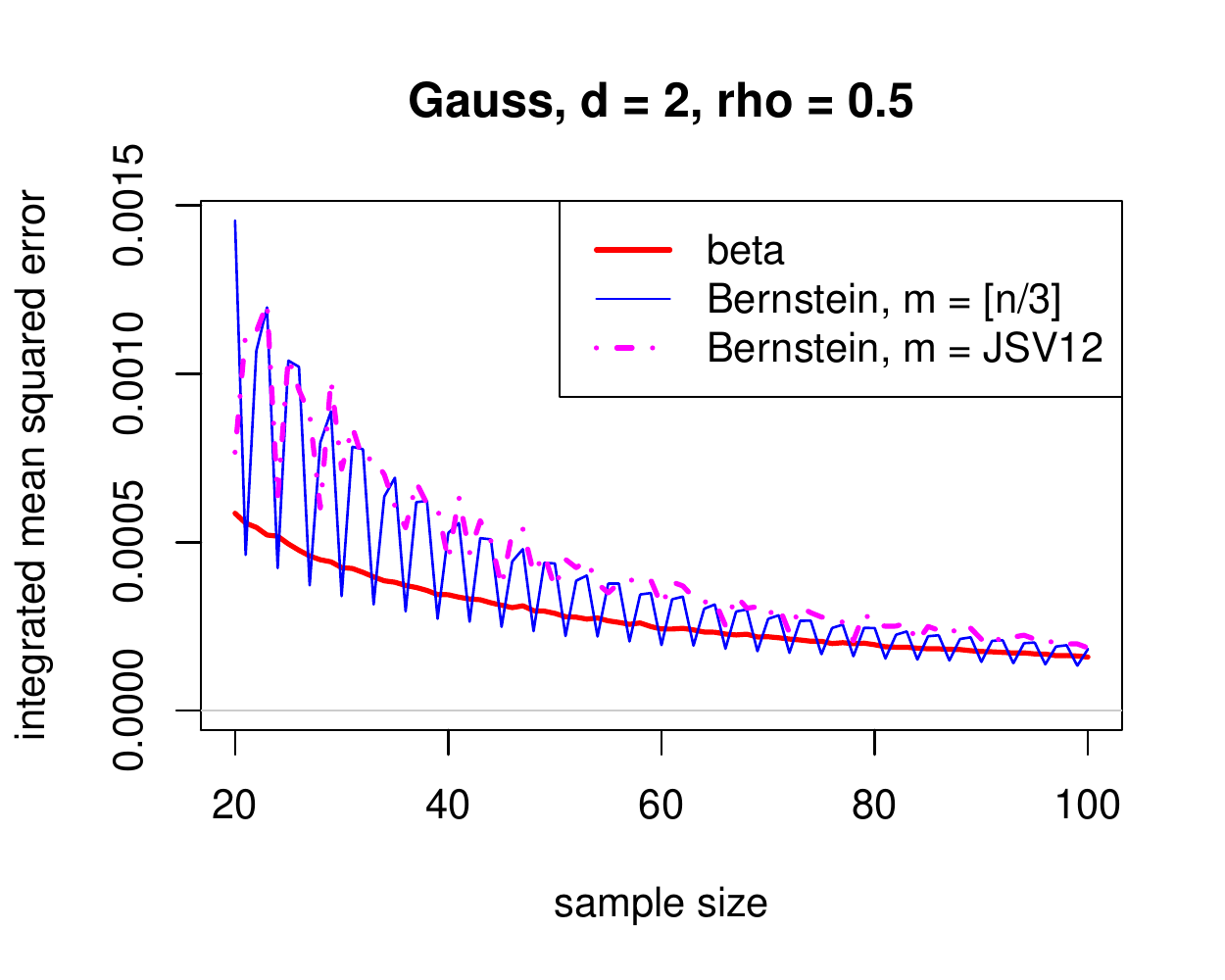}\\
\includegraphics[width=0.33\textwidth]{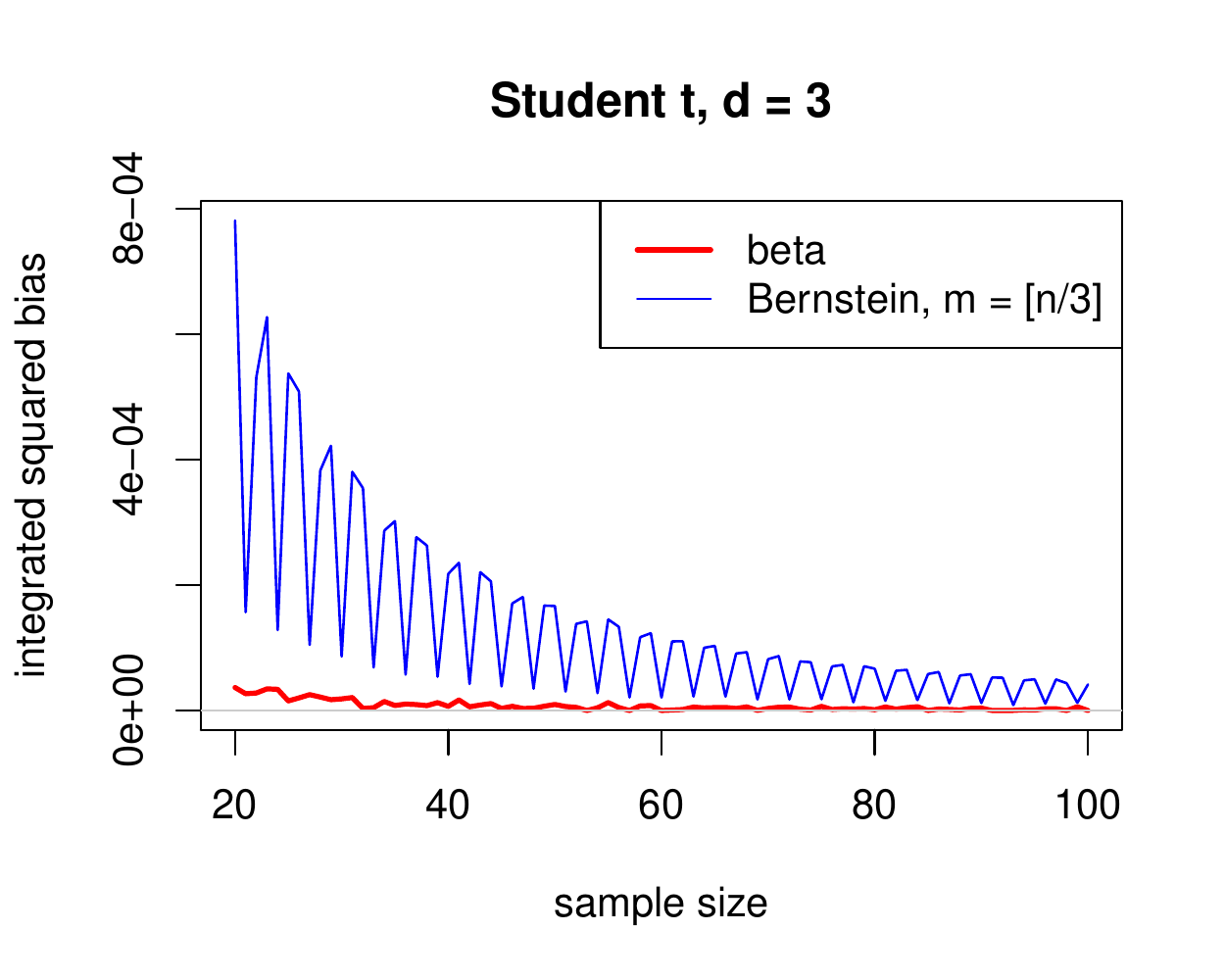}&
\includegraphics[width=0.33\textwidth]{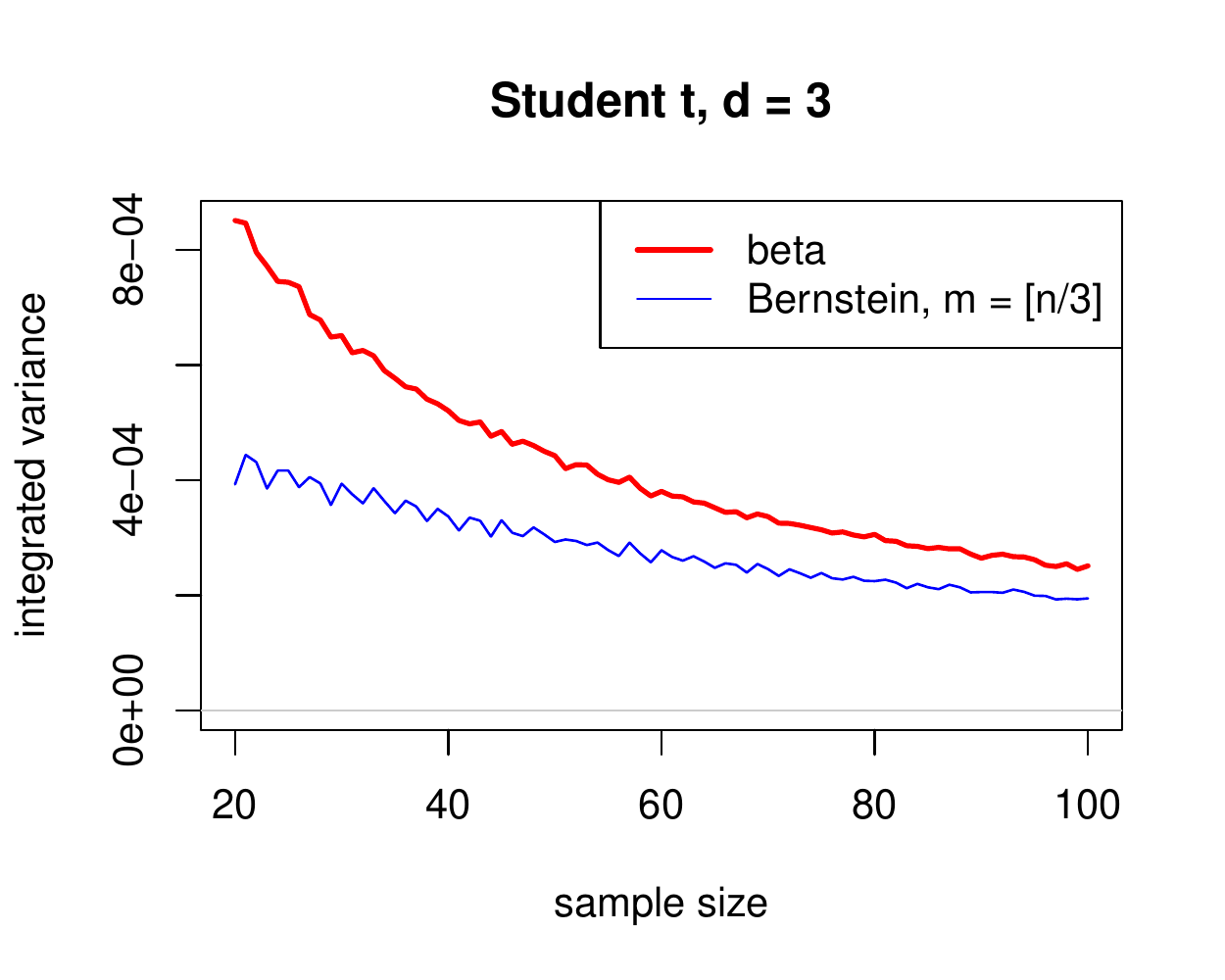}&
\includegraphics[width=0.33\textwidth]{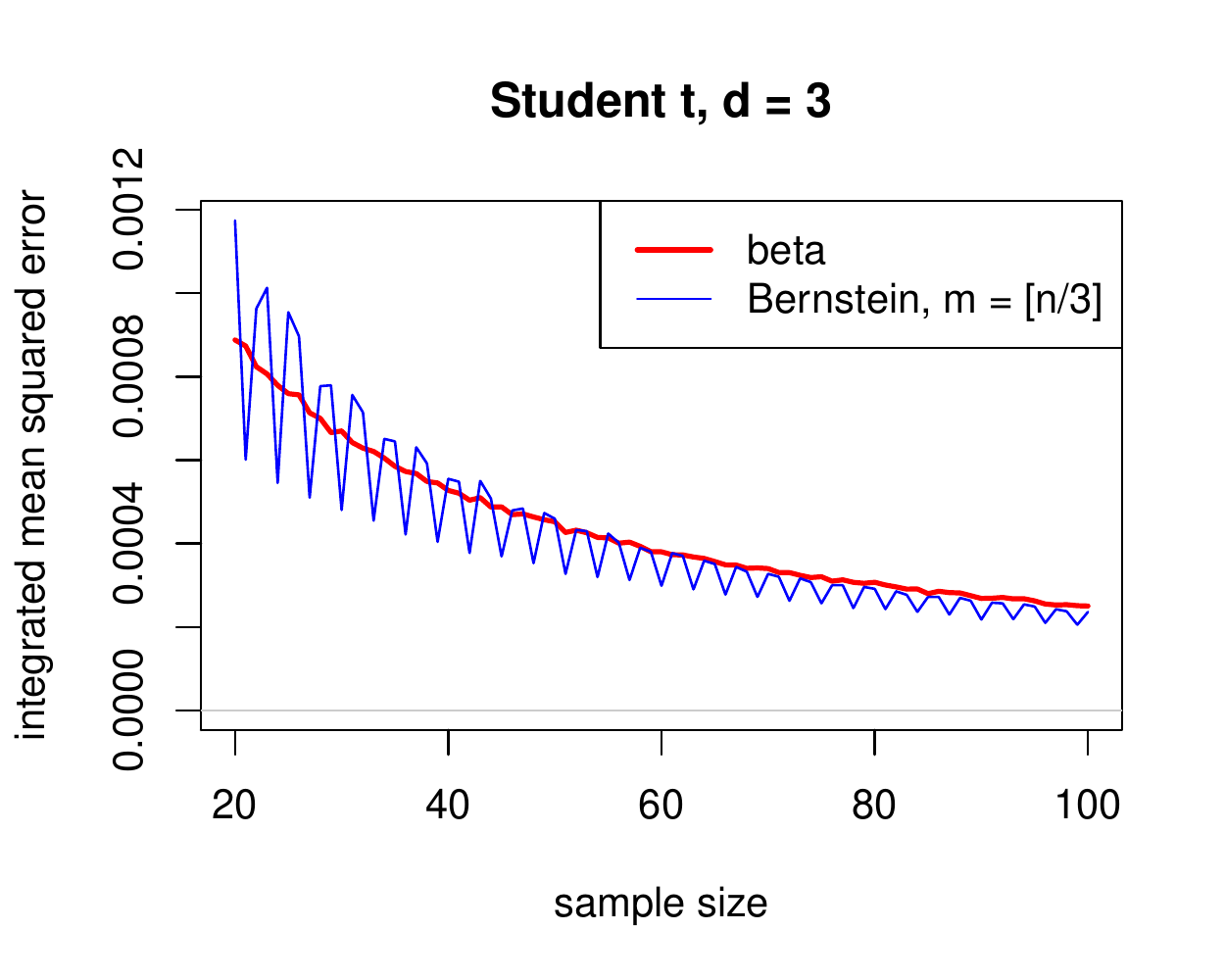}\\
\includegraphics[width=0.33\textwidth]{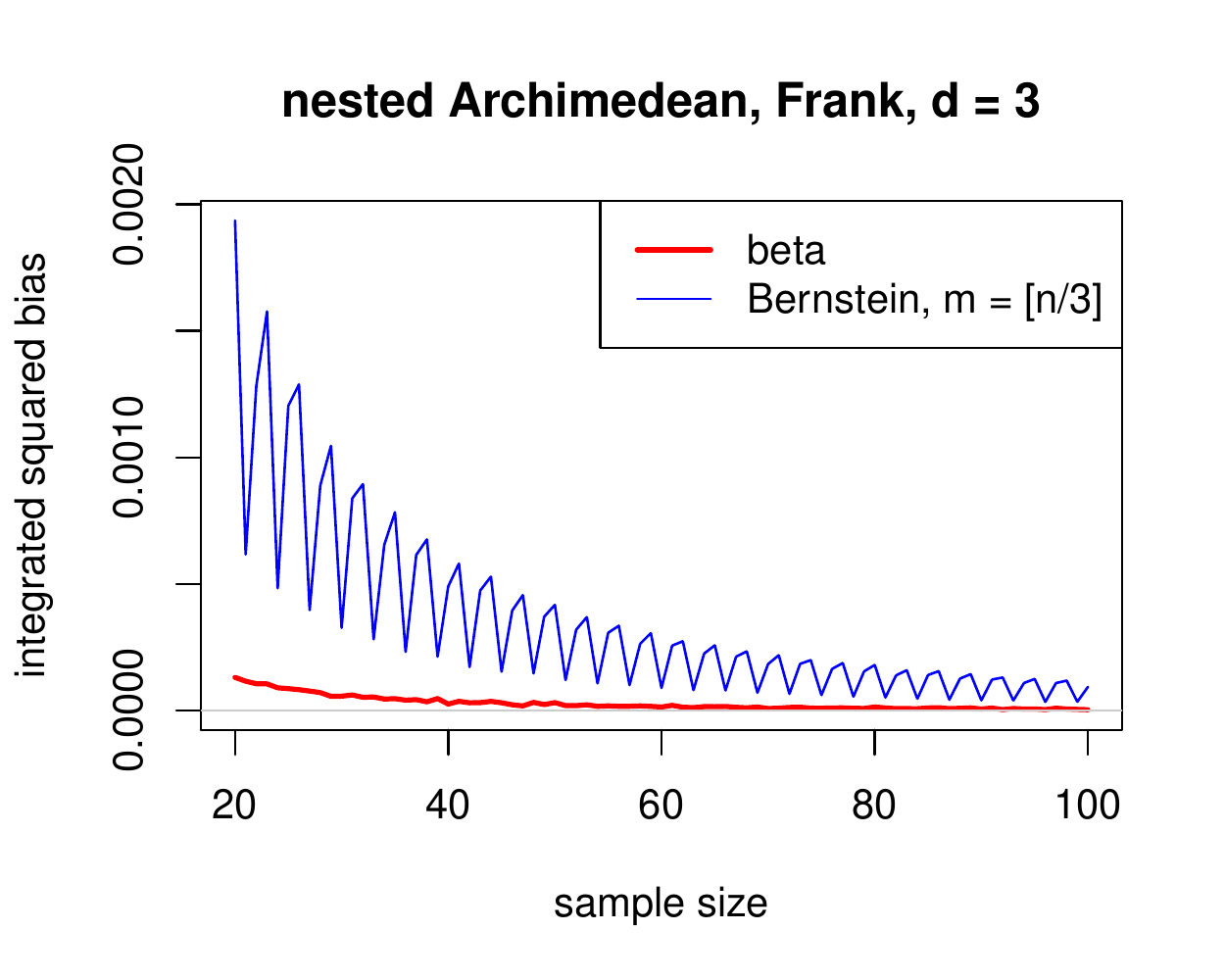}&
\includegraphics[width=0.33\textwidth]{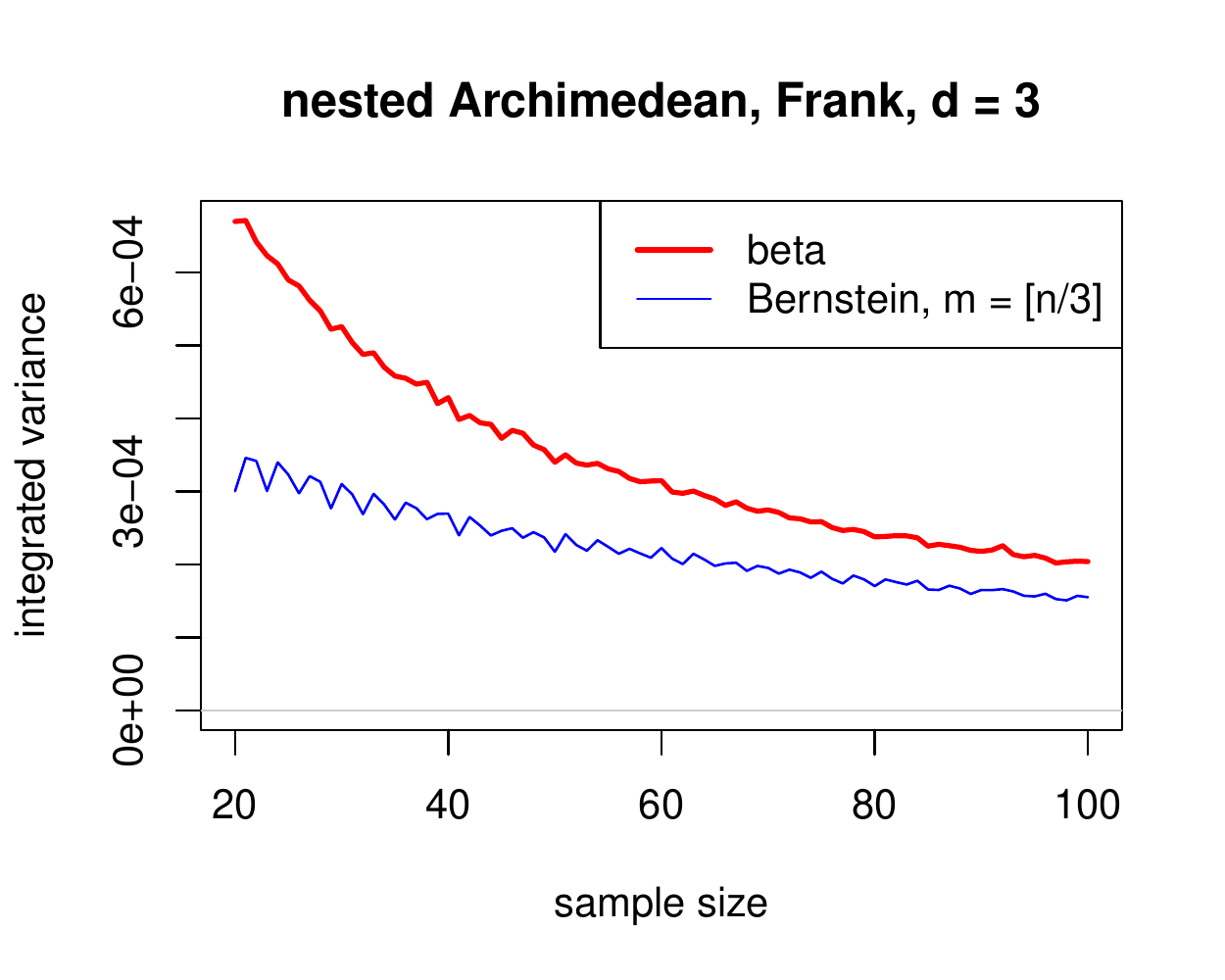}&
\includegraphics[width=0.33\textwidth]{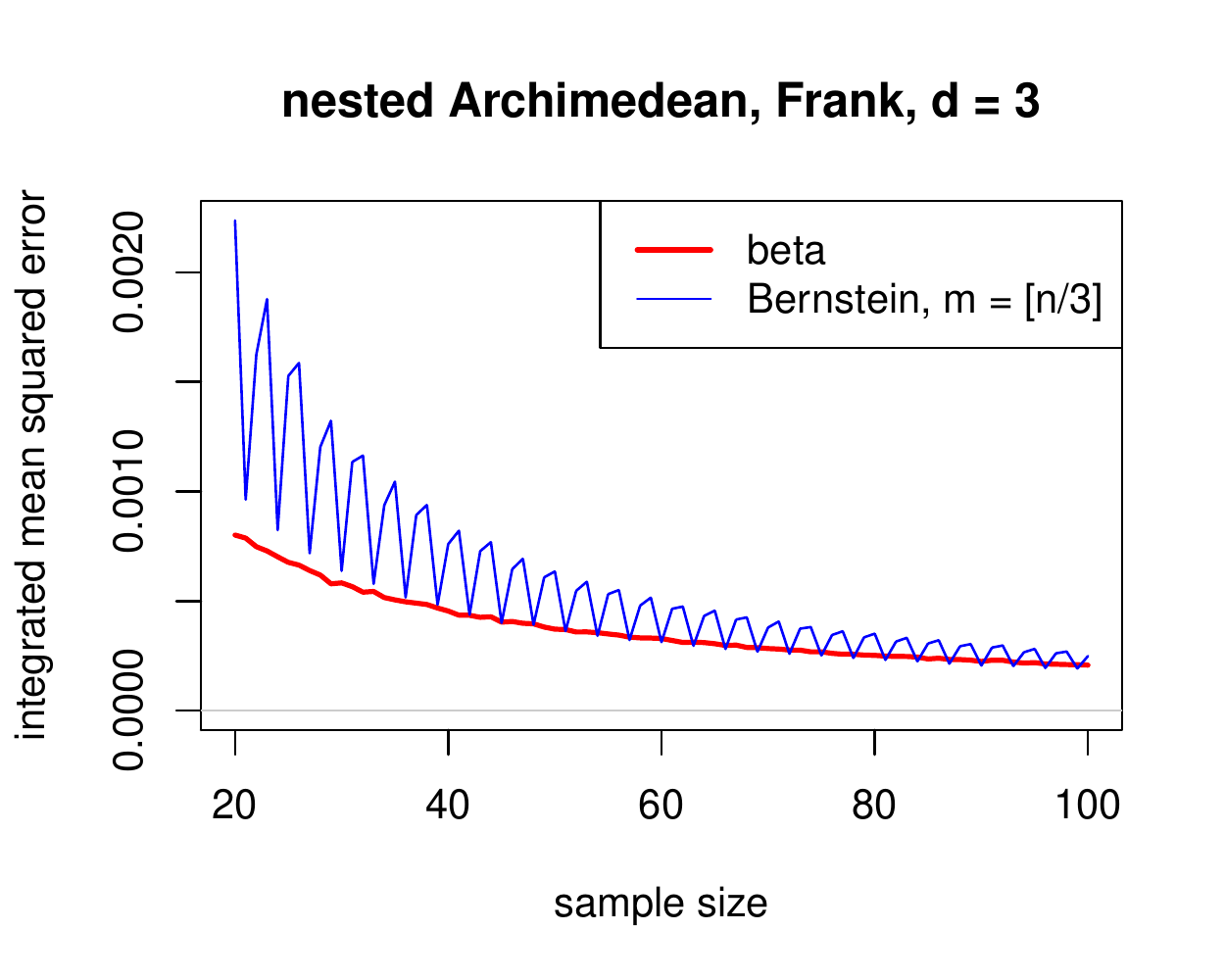}
\end{tabular}
\caption{\label{fig:beta-bern3-bernJSV}
Empirical beta copulas vs empirical Bernstein ($m = \lceil n/3 \rceil$) vs empirical Bernstein ($m$ as in \eqref{eq:JSV}). From the top row: FGM ($\theta=-1$), independence, Gauss ($\rho=0.5$), trivariate Student t, and trivariate nested Archimedean.  Left: integrated squared bias. Middle: integrated variance. Right: integrated mean squared error. See Section~\ref{sec:simul} and Appendix~\ref{app:estim} for details. Based on 20\,000 Monte Carlo replications.}
\end{figure}

\end{document}